\documentclass[12pt]{amsart}
\usepackage[margin=1.0in]{geometry} 
\usepackage{esint,amsthm,amsmath,amssymb,amsfonts,graphics,color,comment,enumerate,psfrag,caption,bbm,bm,enumitem}
\usepackage[hyperpageref]{backref}
\usepackage{mathtools}

\usepackage[colorlinks=true, pdfstartview=FitV, linkcolor=blue, citecolor=blue, urlcolor=blue]{hyperref}
\allowdisplaybreaks

\DeclarePairedDelimiter\floor{\lfloor}{\rfloor}

\numberwithin{equation}{section}
\newtheorem{theorem}{Theorem}[section]

\newtheorem{lemma}[theorem]{Lemma}
\newtheorem{definition}[theorem]{Definition}

\newtheorem{remark}[theorem]{Remark}

\begin{document}
\title[Extension of CZ Type Singular Integrals and commutators]{On extension of Calder\'on-Zygmund Type Singular Integrals and their commutators}

\author[S. Bagchi, R. Garg, J. Singh]
{Sayan Bagchi \and Rahul Garg \and Joydwip Singh} 

\address[S. Bagchi]{Department of Mathematics and Statistics, Indian Institute of Science Education and Research Kolkata, Mohanpur--741246, West Bengal, India.}
\email{sayan.bagchi@iiserkol.ac.in}

\address[R. Garg]{Department of Mathematics, Indian Institute of Science Education and Research Bhopal, Bhopal--462066, Madhya Pradesh, India.}
\email{rahulgarg@iiserb.ac.in}

\address[J. Singh]{Department of Mathematics and Statistics, Indian Institute of Science Education and Research Kolkata, Mohanpur--741246, West Bengal, India.}
\email{js20rs078@iiserkol.ac.in}



\begin{abstract}
Motivated by the recent works \cite{Yu-Jiu-Li-CZ-JFA-2021, Chen-Guo-Extension-CZ-JFA-2021}, we study the following extension of Calder\'on-Zygmund type singular integrals 
$$ T_{\beta}f (x) = p.v. \int_{\mathbb{R}^n} \frac{\Omega(y)}{|y|^{n-\beta}} f(x-y) \, dy, $$ 
for $0 < \beta < n$, and their commutators. We establish estimates of these singular integrals on Lipschitz spaces, Hardy spaces and Muckenhoupt $A_p$-weighted $L^p$-spaces. We also establish Lebesgue and Hardy space estimates of their commutators. Our estimates are uniform in small $\beta$, and therefore one can pass onto the limits as $\beta \to 0$ to deduce analogous estimates for the classical Calder\'on-Zygmund type singular integrals and their commutators.
\end{abstract}

\maketitle

\tableofcontents

\section{Introduction}
Consider the classical Calder\'on-Zygmund type (CZ type) singular integral defined by 
\begin{align*} 
Tf (x) = p.v. \int_{\mathbb{R}^n} \frac{\Omega(y)}{|y|^n} f(x-y) \, dy = \lim_{\varepsilon \to  0} \int_{|y| > \varepsilon} \frac{\Omega(y)}{|y|^n} f(x-y) \, dy, 
\end{align*}
for a suitable class of functions $f$ on $\mathbb{R}^n$. 

In the above, the function $\Omega$ is assumed to satisfy the following conditions: 
\begin{align} \label{conditions:main-Omega-function}
\left\{ \begin{array}{ll}
\Omega \in L^\infty (S^{n-1}), & (L^{\infty}\text{-bounded}) \\ \\ 
\Omega(r x') = \Omega(x'), & \text{(homogeneous of degree $0$)} \\ \\
\int_{S^{n-1}} \Omega(x') \, d \sigma (x') = 0, & \text{ (cancellation)}, 
\end{array} \right.
\end{align}
for every $r> 0$, with $S^{n-1}$ denoting the unit sphere $\left\{ x' \in \mathbb{R}^{n} : |x'| = 1 \right\}$. 

In view of their wide applications to a variety of partial differential equations, CZ type singular integrals have played an extremely significant role in the development of the theory of harmonic analysis too. There is a rich literature on $L^p$ and weighted $L^p$-boundedness of these integrals, and the same can be found in many books (see, for example, \cite{Stein-book-Singular-integrals-1970, Lu-Ding-Yan-book-singular-int-2007}). 

Motivated by the recent works \cite{Yu-Jiu-Li-CZ-JFA-2021, Chen-Guo-Extension-CZ-JFA-2021}, in this paper we consider the following extension of CZ type singular integrals defined, for $0<\beta<n$, by
\begin{equation} \label{def:extension-CZ-type}
T_{\beta}f (x) = p.v. \int_{\mathbb{R}^n} \frac{\Omega(y)}{|y|^{n-\beta}} f(x-y) \, dy. 
\end{equation}
Here, in addition to \eqref{conditions:main-Omega-function}, we also assume that $\Omega$ satisfies the following Dini type condition: 
\begin{align} \label{conditions:Dini-Omega-function}
\int_{0}^{1} \frac{\omega_{\infty}(\delta)}{\delta} \, d \delta < \infty, 
\end{align} 
where $\displaystyle \omega_{\infty}(\delta) := \sup \left\{ \left|\Omega(x') - \Omega(y')\right| : \left| x' - y' \right| \leq \delta, |x'| = |y'| = 1 \right\}.$ 

Yu--Jiu--Li \cite{Yu-Jiu-Li-CZ-JFA-2021} analysed operator $T_\beta$ with $\Omega$ satisfying conditions \eqref{conditions:main-Omega-function} and \eqref{conditions:Dini-Omega-function}, and proved the following $L^q \, (1<q<\infty)$ estimate of $T_\beta$ from which the strong type $(q,q)$ bound of the CZ type singular integral $T$ can be recovered when $\beta \to 0$. 

\begin{theorem} [\cite{Yu-Jiu-Li-CZ-JFA-2021}]
\label{thm:Yu-Jiu-Li-CZ-JFA-2021}
Let $0<\beta_{0}<\frac{1}{2}$ be fixed. Suppose that $\Omega$ satisfies conditions \eqref{conditions:main-Omega-function} and \eqref{conditions:Dini-Omega-function}. Then for any $1<q<\infty$, there exists a constant $C> 0$ such that
\begin{align*} 
\left\|T_{\beta} f\right\|_{{L^{q}}} \leq C\left(\|f\|_{{L^{q}}}+\frac{\beta^{\frac{(q-1) n}{q}}}{\sqrt[q]{(n(q-1)-\beta q)}}\|f\|_{L^{1}} \right)
\end{align*} 
holds true for all $f\in L^{1}(\mathbb{R}^{n})\cap L^{q}(\mathbb{R}^{n})$ and $ 0 < \beta < \min \left\{1-\beta_{0}, \frac{(q-1) n}{q}\right\}$.
\end{theorem}

Soon after the work of \cite{Yu-Jiu-Li-CZ-JFA-2021}, Chen--Guo \cite{Chen-Guo-Extension-CZ-JFA-2021} improved Theorem \ref{thm:Yu-Jiu-Li-CZ-JFA-2021} by getting rid of the restriction of the smoothness of $\Omega$ as well as the condition on $0 < \beta < 1-\beta_0$ for $0<\beta_{0}<\frac{1}{2}$. More precisely, they proved the following result. 
\begin{theorem} [\cite{Chen-Guo-Extension-CZ-JFA-2021}]
\label{thm:Chen-Guo-Extension-CZ-JFA-2021}
Let $\Omega$ satisfy condition \eqref{conditions:main-Omega-function} and $1<q<\infty$. There exists a constant $C> 0$ such that
\begin{align*} 
\left\|T_{\beta} f\right\|_{{L^{q}}} \leq C\left(\|f\|_{{L^{q}}}+\frac{\beta^{\frac{(q-1) n}{q}}}{\sqrt[q]{(n(q-1)-\beta q)}}\|f\|_{L^{1}} \right)
\end{align*} 
holds true for all $f\in L^{1}(\mathbb{R}^{n})\cap L^{q}(\mathbb{R}^{n})$ and $ 0 < \beta < \frac{(q-1) n}{q}$.
\end{theorem}

In \cite{Yu-Jiu-Li-CZ-JFA-2021}, the authors built their analysis based on the so called ``geometric approach" which is inspired by the works \cite{Li-geometric-approach-Calderon-Zygmund-2003,Li-Wang-new-proof-Calderon-Zygmund-2006}. On the other hand, authors of \cite{Chen-Guo-Extension-CZ-JFA-2021} developed their analysis in a very different way. They adapted some ideas from \cite{Duoandikoetxea-Rubio-maximal-singular-int-Inventione-1986, Seeger-Singular-integral-rough-convolution-kernels-JAMS-1996} and made use of the van der Corput lemma and a regularisation of the radial variable of the integral kernel in their proof, and succeeded without needing any smoothness of $\Omega$. The genesis of \cite{Yu-Jiu-Li-CZ-JFA-2021, Chen-Guo-Extension-CZ-JFA-2021} is in the works of \cite{Yu-Jiu-extension-Riesz-transform-2019, Yu-Zheng-Jiu-Remarks-on-well-posedness}, where the authors studied this problem in the context of Riesz transforms, that is, for operators $T_\beta$ corresponding to $\Omega(x) = x_j / |x|$, with $j = 1, 2, \ldots, n$. 

An important aspect to study operators of the form $T_\beta$ is not just its formal connection with the CZ operator $T$, that we can recover $T$ from $T_\beta$ (in an appropriate sense) by taking $\beta \to 0$, but also the fact that good quantitative estimates on these operators $T_\beta$ provide a crucial tool in the analysis of several classical operators. For example, such singular integrals appear in the approximation of the surface quasi-geostrophic (SQG) equation by the generalised SQG equation, as explained beautifully in \cite{Yu-Jiu-extension-Riesz-transform-2019, Yu-Zheng-Jiu-Remarks-on-well-posedness, Chen-Guo-Extension-CZ-JFA-2021}. In our work, we borrow ideas from \cite{Yu-Jiu-Li-CZ-JFA-2021,Chen-Guo-Extension-CZ-JFA-2021}, but in our proofs we shall always need some smoothness condition on $\Omega$. Altogether, in the spirit of Theorem \ref{thm:Yu-Jiu-Li-CZ-JFA-2021}, the main aim of this paper is to develop results for singular integrals $T_\beta$ (given by \eqref{def:extension-CZ-type}) as well as their commutators on various function spaces. We shall now proceed to describe our results in more detail. 

In some of our results, we shall also need the following stronger Dini type condition: 
\begin{align} \label{conditions:alpha-Dini-Omega-function}
\int_{0}^{1} \frac{\omega_{\infty}(\delta)}{\delta^{1+\alpha}} \, d \delta < \infty
\end{align} 
for some $\alpha \in (0,1]$. It is not difficult to verify that the function $\Omega(x) = x_j / |x|$ with $j = 1, 2, \ldots, n$, satisfies both of the conditions \eqref{conditions:main-Omega-function} and \eqref{conditions:alpha-Dini-Omega-function}.

Let us first discuss our results on Hardy spaces $H^p (\mathbb{R}^n)$ for $0 < p \leq 1$. A detailed history of classical results on Hardy spaces can be found in \cite{Stein-book-Harmonic-Analysis-1993,  Lu-book-4-Lect-Hp-1995, Grafakos-Modern-Fourier-Analysis}. As a special case of Theorems 2 and 3 of \cite{Ding-Lu-hom-frac-Hardy-Tohoku-2000}, it is known that if $\Omega$ satisfies conditions \eqref{conditions:main-Omega-function} and \eqref{conditions:alpha-Dini-Omega-function} for some $0 < \alpha \leq 1$, then 
\begin{align*}
& \| T_{\alpha} f \|_{L^{q}} \leq C_{\alpha} \|{f}\|_{H^p} \quad \text{for } \, 0<\alpha<1, \frac{n}{n+\alpha} \leq p <1, \, \frac{1}{q}=\frac{1}{p}-\frac{\alpha}{n}, \\ 
\text{and} \quad & \| T_{\beta} f \|_{H^{q}} \leq C_{\beta} \|{f}\|_{H^p} \quad \text{for } \, 0<\beta<\frac{1}{2}, \, \frac{n}{n+\alpha} < p \leq \frac{n}{n+\beta}, \, \frac{1}{q}=\frac{1}{p}-\frac{\beta}{n}. 
\end{align*}

But, in proving these estimates, one only gets that the implicit constants $C_{\alpha}$ or $C_{\beta} \to \infty$ as $\alpha$ or $\beta \to 0$, and therefore one cannot deduce boundedness of operator $T$.

On the other hand, it is well known that the dual of the Hardy space $H^p$ ($0 < p \leq 1$) is a Campanato space $\mathcal{L}_{\frac{1}{p}-1,r,s}$, which is equivalent to the Lipschitz space $Lip_{n(\frac{1}{p}-1)}$ whenever $\frac{n}{n+1} < p < 1$ (see, Theorems \ref{thm:dual-Hp=Campanato} and \ref{lem:Campanato=Lipschitz}). It is therefore natural to study an analogue of Theorem \ref{thm:Yu-Jiu-Li-CZ-JFA-2021} on Lipschitz spaces. 

Note that given $0 < \gamma < \alpha \leq 1$, there is a unique $p \in (\frac{n}{n+\alpha},1)$ such that $\gamma = n(\frac{1}{p}-1)$. Following is our main result for CZ type operators on Lipschitz spaces. 

\begin{theorem} \label{thm:CZ-Lipschitz-space}
Let $\Omega$ satisfy conditions \eqref{conditions:main-Omega-function} and \eqref{conditions:alpha-Dini-Omega-function} for some $0 < \alpha \leq 1$. If $\frac{n}{n+\alpha} < p < 1 < q < \infty$ are such that $\frac{1}{p}+\frac{1}{q}=2$, then there exists a constant $ C>0$ such that 
\begin{align*}
\left\|T_{\beta} f\right\|_{Lip_{n(\frac{1}{p}-1)}} \leq C \left( \|f\|_{Lip_{n(\frac{1}{p}-1)}}+\frac{\beta^{\frac{(q-1) n}{q}}}{\sqrt[q]{(n(q-1)-\beta q)}}\|f\|_{BMO} \right).
\end{align*}
holds true for all $f \in Lip_{n(\frac{1}{p}-1)} \cap BMO $ and $0 < \beta < \frac{(q-1) n}{q}$. 
\end{theorem}

Building on the ideas of the proof of Theorem \ref{thm:CZ-Lipschitz-space}, we also establish weighted $L^p$-estimates and Hardy space estimates of $T_\beta$. We state and prove these results (Theorem \ref{thm:CZ-weighted-Lp-estimate} and \ref{thm:CZ-Hardy-estimate}) in Subsections \ref{subsec:Weighted-CZ} and \ref{subsec:Hardy-CZ}. 

Next, we pursue an analogous study for commutator operators. With $\Omega$ satisfying \eqref{conditions:main-Omega-function}, the commutator $[b, T]$ generated by a CZ type singular integral $T$ and a measurable function $b$ on $\mathbb{R}^n$ is defined by 
\begin{align*}
[b,T]f(x) & : = b(x) Tf(x) - T(bf)(x) = p.v. \int_{\mathbb{R}^n} \frac{\Omega(x-y)}{|x-y|^{n}} [b(x)-b(y)] f(y) \, dy. 
\end{align*}

Early fundamental work on these commutator operators is due to Coifman--Rochberg--Weiss in \cite{Coifman-Rochberg-Weiss-AnnalsMath-1976}, and these are often referred as CRW-type commutators. These commutators are known to be extremely useful in studying regularity of solutions of several second order elliptic partial differential equations. Concerning their $L^p$-boundedness (for $1<p<\infty$), it was shown in \cite{Coifman-Rochberg-Weiss-AnnalsMath-1976} that if $\Omega$ is also a Lipschitz function on $S^{n-1}$, that is, $|\Omega (x) - \Omega (y) | \lesssim |x - y|$ for $|x|$ = $|y| = 1$, then the commutator operator $[b,T]$ is $L^p$-bounded if and only if $b\in BMO$. Janson \cite{Janson-mean-osc-commutator-ArkMath-1978} extended the study initiated by \cite{Coifman-Rochberg-Weiss-AnnalsMath-1976}, showing that with $1 < p < q < \infty$ and $\frac{1}{p}- \frac{1}{q} = \frac{\gamma}{n}$, the commutator operator $[b, T]$ is $(L^p, L^q)$-bounded if and only if $b$ belongs to the Lipschitz space $Lip_{\gamma}(\mathbb{R}^n)$. For more details on this subject, we refer to \cite{Lu-Ding-Yan-book-singular-int-2007}. Recall also that a closely related type of commutator was studied by Chanillo \cite{Chanillo-note-commutators-IUMJ-1982}. For $0<\beta<n$, consider the Riesz potential operator $I_{\beta}$ defined by $\displaystyle I_{\beta}f(x)= \int_{\mathbb{R}^n} \frac{f(y)}{|x-y|^{n-\beta}} \, dy.$ It was shown in \cite{Chanillo-note-commutators-IUMJ-1982} that if $1 < p < q < \infty$ and $\frac{1}{p}- \frac{1}{q} = \frac{\beta}{n}$, then $[b,I_{\beta}]$ is $(L^p, L^q)$-bounded if and only if $b\in BMO(\mathbb{R}^n) $. Results of \cite{Janson-mean-osc-commutator-ArkMath-1978, Chanillo-note-commutators-IUMJ-1982} were further explored and generalised by Paluszy\'{n}ski \cite{Paluszynski-Besov-via-commutator-IUMJ-1995} for Lipschitz functions $b$. 

Now, with $\Omega$ satisfying \eqref{conditions:main-Omega-function}, we consider the following extension of the commutator of the CZ type operator, defined by
\begin{align} \label{def:extension-commutator-operator}
[b,T_{\beta}]f(x) & : = b(x)T_{\beta}f(x)-T_{\beta}(bf)(x) = p.v.\int_{\mathbb{R}^n}\frac{\Omega(x-y)}{|x-y|^{n-\beta}} [b(x)-b(y)] f(y) \, dy
\end{align}
for any $0<\beta < n$, so that formally, $[b,T_{\beta}]$ becomes CRW-type commutator when $\beta=0$.

It is known (see, for example, Chapter 3 of \cite{Lu-Ding-Yan-book-singular-int-2007}) that with $b \in BMO$, 
\begin{align*} 
\|[b,T_{\beta}]f\|_{L^q} \leq C_{\beta} \|f\|_{L^p}, 
\end{align*}
for $0<\beta < n$ and $1 < p < q < \infty$ such that $\frac{1}{p} - \frac{1}{q} = \frac{\beta}{n}$. 

On the other hand, one can show that $|[b,T_{\beta}]f(x)| \leq C \|b\|_{Lip_{\gamma}} I_{\beta+\gamma}(|f|)(x)$ whenever $b \in Lip_{\gamma}(\mathbb{R}^n)$, which implies that 
\begin{align*} 
\|[b,T_{\beta}]f\|_{L^q} \leq C_{\beta} \|b\|_{Lip_{\gamma}} \|f\|_{L^p}, 
\end{align*} 
for $0<\beta<n$, $0<\gamma<1$, and $1 < p < q < \infty$ such that $\frac{1}{p} - \frac{1}{q} = \frac{\gamma}{n}+\frac{\beta}{n}$. 

Note that in either of the above two estimates, the constant $C_{\beta} \to \infty$ as $\beta \to 0$. Analogous to Theorem \ref{thm:Yu-Jiu-Li-CZ-JFA-2021}, following is our main result for these commutators. 
\begin{theorem} \label{thm:commutator-b-BMO-Lp-estimate}
Let $\Omega$ satisfy conditions \eqref{conditions:main-Omega-function} and \eqref{conditions:Dini-Omega-function}. If $1 < r < p < \infty$ and $0 < l < n$ are such that  $\frac{1}{p} = \frac{1}{r} - \frac{l}{n}$, then there exists a constant $C > 0$ such that 
\begin{align*}
\|[b,T_{\beta}] f\|_{L^p} & \leq C \|b\|_{BMO} \left(\|f\|_{L^p} + \frac{\beta^{n(1-\frac{1}{p})}}{\sqrt[p]{n(p-1)-\beta p}} \|f\|_{L^1} + \frac{\beta^{l}}{(l-\beta)^{1- \frac{l}{n}}} \|f\|_{L^r} \right)
\end{align*}
holds true for all $b\in BMO$, $f \in L^p(\mathbb{R}^n) \cap L^1(\mathbb{R}^n)$, and $0<\beta <l$. 
\end{theorem}

Concerning functions $b$ belonging to Lipschitz spaces $Lip_{\gamma}(\mathbb{R}^n)$, we suitably modify ideas of the proof of Theorem \ref{thm:commutator-b-BMO-Lp-estimate} and study Lebesgue and Hardy space boundedness of commutator operators $[b,T_{\beta}]$. Before stating our results, let us mention that it was shown in \cite{Lu-Wu-Yang-commutators-Hardy-2002-Science-China} that $[b,T]$ is $(H^p,L^q)$-bounded for every $ b \in Lip_{\gamma}(\mathbb{R}^n)$, $0<\gamma \leq 1$, $\frac{n}{n+\gamma}<p\leq 1$ and $\frac{1}{q} = \frac{1}{p} - \frac{\gamma}{n}$. We have the following $(L^p, L^q)$ and $(H^p, L^q)$-estimates for $[b,T_{\beta}]$. 
\begin{theorem} \label{thm:commutator-b-Lipschitz-Lp-Lq-estimate}
Let $\Omega$ satisfy conditions \eqref{conditions:main-Omega-function} and \eqref{conditions:Dini-Omega-function}. If $1<p<q<\infty$ and $0<\gamma<1$ be such that $\frac{1}{q} = \frac{1}{p} - \frac{\gamma}{n}$, then there exists a constant $ C>0$ such that 
\begin{align*}
\|[b,T_{\beta}]f \|_{L^q} \leq C \|b\|_{Lip_{\gamma}} \left( \|f \|_{L^p} + \frac{\beta^{n(1-\frac{1}{q})-\gamma}}{\sqrt[q]{n(q-1)-(\beta-\gamma)q}}\|f\|_{L^1} \right)
\end{align*}
holds true for all $b \in Lip_{\gamma}(\mathbb{R}^n)$, $f \in L^p(\mathbb{R}^n) \cap L^1(\mathbb{R}^n)$, and $0< \beta < n(1-\frac{1}{q}) - \gamma$. 
\end{theorem} 

\begin{theorem} \label{thm:commutator-b-Lipschitz-Hp-Lq-estimate}
Let $\Omega$ satisfy conditions \eqref{conditions:main-Omega-function} and \eqref{conditions:alpha-Dini-Omega-function} for some $0 < \alpha < 1,$ and let $b \in Lip_{\alpha}(\mathbb{R}^n)$. If $\frac{n}{n+\alpha} < p \leq 1 < q < \infty$ be such that $\frac{1}{q}=\frac{1}{p}-\frac{\alpha}{n}$, then there exists a constant $ C>0$ such that 
\begin{align*} 
\| [b,T_{\beta}] f \|_{L^q} \leq C \|b\|_{Lip_{\alpha}} \left( \|f \|_{H^p} + \frac{\beta^{n(1-\frac{1}{q})-\alpha}}{\sqrt[q]{n(q-1)-(\beta-\alpha)q}}\|f\|_{L^1} \right)
\end{align*} 
holds true for all $ f \in H^p(\mathbb{R}^n) \cap L^1(\mathbb{R}^n)$, and $0< \beta < n (1-\frac{1}{q}) - \alpha$. 
\end{theorem}

In our proofs, as usual, we shall decompose the operator $T_{\beta}$ into two parts using a suitable cut-off function: one part supported near the origin and the other one supported away from the origin. For the same, fix a function $\chi \in C_{c}^{\infty} (\mathbb{R})$ such that $0 \leq \chi(s) \leq 1$, $|\chi^{\prime}(s)|\leq 2$, and 
\begin{align} \label{def:cut-off-function}
\chi(s)  = \left\{\begin{array}{cc}
1  &\textup{ if } |s| \leq 1; \\
0  &\textup{ if } |s| > 2, 
\end{array} \right.
\end{align}
and write $\chi_\lambda(s) = \chi(\lambda s)$ for $\lambda > 0$. 

With that, we decompose $T_{\beta} = T_1 + T_2$, where 
\begin{align} \label{def:operator-decomposition}
T_1f (x) = p.v.\int_{\mathbb{R}^n}\frac{\Omega(y)}{|y|^{n-\beta}}\chi_\beta (|y|)f(x-y) \, dy. 
\end{align}

\medskip The organisation of the paper is as follows. We recollect the relevant preliminary details in Section \ref{sec:prelim}. Proof of Theorem \ref{thm:CZ-Lipschitz-space} in developed in Subsection \ref{subsec:CZ-Lipschitz}, whereas 
a variant of Hardy space estimates and some weighted $L^p$-estimates (resp. Theorems \ref{thm:CZ-Hardy-estimate} and \ref{thm:CZ-weighted-Lp-estimate}) are discussed in Subsections \ref{subsec:Hardy-CZ} and \ref{subsec:Weighted-CZ}. Section \ref{sec:proofs-commutator} is devoted to the study of commutator operators. We prove Theorem \ref{thm:commutator-b-BMO-Lp-estimate} in Subsection \ref{subsec:commutator-BMO}, and finally proofs of Theorems \ref{thm:commutator-b-Lipschitz-Lp-Lq-estimate} and \ref{thm:commutator-b-Lipschitz-Hp-Lq-estimate} are given in Subsections \ref{subsec:commutator-Lp-Lq-spaces} and \ref{subsec:commutator-Hp-Lq-spaces}, respectively. 

\medskip \noindent \textbf{Notations:}
For any cube $Q$ (resp. ball $B$) in $\mathbb{R}^n$ and $r>0$, we denote by $r Q$ the cube (resp. by $r B$ the ball) with the same center as $Q$ (resp. $B$) and the side-length $rl$ (resp. diameter $rd$), where $l$ (resp. $d$) is the side-length (resp. diameter) of $Q$ (resp. $B$). For a Lebesgue measurable set $E \subset \mathbb{R}^n$, we denote its its Lebesgue measure by $|E|$, and write ${\mathbbm{1}}_E$ for its characteristic function. Also, given a Lebesgue measurable function $f$ on $\mathbb{R}^n$, its support set $supp (f)$ is the closure of the set $\{x\in \mathbb{R}^n : f(x) \neq 0\}$, and $f_E$ denotes the average of $f$ over $E$ given by $\displaystyle f_E = \frac{1}{|E|}\int_{E} f(y) \, dy$.

For $A,B>0$, by the expression $A \lesssim B$ we mean $A \leq C B$ for some $C>0$. Also, we write $A \lesssim_{\epsilon} B$ when the implicit constant $C$ may depend on a parameter like $\epsilon$. By $A \sim B$, we mean that $A$ is equivalent to $B$, that is, $A \lesssim B$ and $B \lesssim A$. 

Throughout the article, $C$, $C_\epsilon$ etc will denote  constants whose values might differ from place to place.

%%%%%%%%%%%%%%%%%%%%%%%%%%%%%%%%%%%
%%%%%%%%%%%%%%%%%%%%%%%%%%%%%%%%%%%
%%%%%%%%%%%%%%%%%%%%%%%%%%%%%%%%%%%
%%%%%%%%%%%%%%%%%%%%%%%%%%%%%%%%%%%
%%%%%%%%%%%%%%%%%%%%%%%%%%%%%%%%%%%
%%%%%%%%%%%%%%%%%%%%%%%%%%%%%%%%%%%
%%%%%%%%%%%%%%%%%%%%%%%%%%%%%%%%%%%
%%%%%%%%%%%%%%%%%%%%%%%%%%%%%%%%%%%
%%%%%%%%%%%%%%%%%%%%%%%%%%%%%%%%%%%
%%%%%%%%%%%%%%%%%%%%%%%%%%%%%%%%%%%
%%%%%%%%%%%%%%%%%%%%%%%%%%%%%%%%%%%

\section{Preliminaries} \label{sec:prelim}
Here, we collect the relevant preliminary details related to the topics studied in this article.  

%%%%%%%%%%%%%%%%%%%%%%%%%%%%%%%%%%%
%%%%%%%%%%%%%%%%%%%%%%%%%%%%%%%%%%%

\subsection{Hardy Spaces} \label{subsec:Hardy-spaces}
Let us begin with definitions of \emph{atoms} and \emph{molecules}. For details, we refer to \cite{Taibleson-Weiss-molecular-Hardy-Asterisque-1980, Stein-book-Harmonic-Analysis-1993, Lu-book-4-Lect-Hp-1995}.

\begin{definition}[Atoms] \label{def:atoms}
Let $0<p\leq 1 \leq l \leq \infty, \, p <l$, and $s \geq \floor*{n(\frac{1}{p}-1)}$. A $(p,l,s)$-atom centered at $x_{0}$ is a function $a \in L^{l}(\mathbb{R}^n)$ supported on a ball $B \subset \mathbb{R}^n$ with centre $x_{0}$ and satisfying, 
\begin{enumerate}[label=(\roman*), start=1]
\item $\displaystyle \|a\|_{L^l} \leq |B|^{\frac{1}{l}-\frac{1}{p}}$, 

\item $\displaystyle \int_{\mathbb{R}^n} a(x) x^{\alpha} \, dx =0$, for all $|\alpha| \leq s$. 
\end{enumerate}
\end{definition}

\begin{definition}[Molecules] \label{def:molecules}
Let $0 < p \leq 1 \leq l \leq \infty$ , $p < l$ , $s \geq \floor*{n(\frac{1}{p}-1)}$, and $\varepsilon > max\{\frac{s}{n},\frac{1}{p}-1 \}$. Set $a_0 = 1-\frac{1}{p}+\varepsilon$ and $b_0 = 1-\frac{1}{l}+\varepsilon$. 
A $(p,l,s,\varepsilon)$-molecule centered at $x_0$ is a function $F$ such that $F \in L^l(\mathbb{R}^n)$ and $F(x)|x|^{n b_0}\in L^l(\mathbb{R}^n)$ satisfying,
\begin{enumerate}[label=(\roman*), start=1]
\item $\displaystyle \mathcal{N}_{l}(F) : = \left\| F \right\|_{L^l}^{\frac{a_0}{b_0}} \, \left\| F \right|x-x_0|^{n b_0} \|_{L^l}^{1-\frac{a_0}{b_0}} < \infty$, 

\item $\displaystyle \int_{\mathbb{R}^n}F(x)x^\alpha \, dx = 0$, for all $|\alpha| \leq s$. 
\end{enumerate}
\end{definition}

Now, Hardy spaces $H^{p}$ are defined as follows. 
\begin{definition} \label{def:Hardy-space}
Given $0 < p \leq 1$, the Hardy space $H^{p}$ consists of all tempered distributions $f$ admitting a decomposition $f=\sum_{j}\lambda_{j}a_{j}$, where $a_{j}$ are $(p,l,s)$-atoms and $\sum_{j}|\lambda_{j}|^p < \infty$. We also define $\displaystyle \|f\|_{H^p} := \inf \left(\sum_{j}|\lambda_{j}|^p \right)^{1/p},$ with the infimum being taken over all admissible representations $f=\sum_{j}\lambda_{j}a_{j}$. 
\end{definition}

It is known that the $(H^{p_1}, H^{p_2})$-boundedness for a linear operator $T$ can be proved by studying its action on \emph{atoms} and \emph{molecules}. More precisely, the following result holds true, and we refer to \cite{Lu-book-4-Lect-Hp-1995} for a proof. 

\begin{lemma} \label{lem:Hp1-Hp2-via-molecules}
Let $0 < p_1 \leq p_2 \leq 1$ and $l_1, l_2 \in (1, \infty)$. If for any $(p_1,l_1,s_1)$-atom $a$, $Ta$ is $(p_2,l_2,s_2,\varepsilon)$-molecule, and satisfies $\mathcal{N}_{l_2}(Ta)\leq C$, where C is independent of $a$, then $T$ is $(H^{p_1},H^{p_2})$-bounded. 
\end{lemma}

%%%%%%%%%%%%%%%%%%%%%%%%%%%%%%%%%%%
%%%%%%%%%%%%%%%%%%%%%%%%%%%%%%%%%%%

\subsection{BMO, Campanato and Lipschitz spaces} 
\label{subsec:prelims-BMO-Campanato-Lipschitz-spaces}
Let $f \in L^1_{loc}(\mathbb{R}^n)$. The Hardy-Littlewood maximal function $Mf$ and the Fefferman--Stein sharp maximal function $M^{\sharp}f$ of $f$ are given by 
\begin{align*} 
Mf (x) := \sup_{x \in Q} \frac{1}{|Q|}\int_Q |f(y)| \, dy \quad \text{and} \quad M^{\sharp}f(x) := \sup_{x \in Q} \frac{1}{|Q|} \int_{Q} |f(y) - f_Q| \, dy, 
\end{align*} 
where the supremum is taken over all cubes $Q \subset \mathbb{R}^n$ containing $x$. 

Given a cube $Q$ containing a point $x$, let $Q'$ be the cube centered $x$ with side-length double than that of $Q$, then $Q \subset Q'$ and $|Q'| \sim |Q|$. Moreover, for any $c \in \mathbb{C}$, it is easy to see that 
\begin{align*} 
\frac{1}{|Q|} \int_{Q} |f(y) - f_Q| \, dy \leq \frac{2}{|Q|} \int_{Q} |f(y) - c| \, dy \lesssim \frac{1}{|Q'|} \int_{Q'} |f(y) - c| \, dy,
\end{align*}
so that 
\begin{align} \label{ineq:sharp-maximal-function-pointwise} 
M^{\sharp}f(x) \lesssim \sup_{x \in Q} \frac{1}{|Q|} \int_{Q} | f(y) - c | \, dy, 
\end{align} 
where the supremum is taken over all cubes $Q \subset \mathbb{R}^n$ ``centered" at $x$. 

Next, we recall the definition of the BMO space: 
$$ BMO := \{f \in L^1_{loc}(\mathbb{R}^n) : M^{\#}f \in L^{\infty} \}, $$ 
with the norm of $f \in BMO$ (upto a difference by constants) given by 
$$\|f\|_{BMO} := \|M^{\#}f\|_{L^{\infty}}.$$ 
The John-Nirenberg inequality (see, Remark 2.4.1 on Page 120 in \cite{Lu-Ding-Yan-book-singular-int-2007}) implies that 
$$\|f\|_{BMO} \sim \sup_{Q}\left(\frac{1}{|Q|}\int_{Q}|f(x)-f_Q|^p \, dx \right)^{\frac{1}{p}},$$
for any $1<p<\infty$. 

We now define Campanato spaces which can also be seen as a generalisation of BMO spaces. We refer to \cite{Lu-book-4-Lect-Hp-1995} for more details on these spaces. 
\begin{definition} \label{def:Campanato-space}
Given $\lambda\in [0,\infty), r \in [1,\infty],$ and $s\in \mathbb{Z}_+$, the Campanato space $\mathcal{L}_{\lambda,r,s}$ is defined to be the set of all $f\in L_{loc}^1(\mathbb{R}^n)$ such that $$\|f\|_{\mathcal{L}_{\lambda,r,s}} := \underset{ball \ B\subseteq \mathbb{R}^n}{sup}\frac{1}{|B|^{\lambda}} \left\{\frac{1}{|B|}\int_{B}\left|f(x)-P_{B}^{(s)}(f)(x) \right|^r dx \right\}^{1/r} < \infty, $$ 
where $P_{B}^{(s)}(f)$ denotes the unique polynomial of degree not greater than $s$ such that, for any $h(x) \in \mathcal{P}_{s}$, the set of all polynomials with its degree $\leq s$, $$ \int_{B}\left\{f(x)-P_{B}^{(s)}(f)(x) \right\} h(x) \, dx= 0. $$
\end{definition}

In particular, when $r=1$ and $\lambda = 0$, it turns out that $\mathcal{L}_{0,1,s} = BMO$. In general, the following duality result holds ture (see \cite{Lu-book-4-Lect-Hp-1995}). 
\begin{theorem} \label{thm:dual-Hp=Campanato}
Suppose $0 < p \leq 1$, $1 \leq r \leq \infty$, $\frac{1}{r}+ \frac{1}{r^{\prime}} = 1$, $r^\prime \neq p$ and $s \geq \floor*{n(\frac{1}{p}-1)}$. Then, the dual space of $H^p$ is $\mathcal{L}_{\frac{1}{p}-1,r,s}$ .
\end{theorem}

Next, we consider Lipschitz spaces which are defined as follows. For more details on these spaces, we refer to \cite{Lu-Wu-Yang-commutators-Hardy-2002-Science-China}. 
\begin{definition} \label{def:lipschitz-functions-space}
For $0<\gamma < 1$, the Lipschitz space $Lip_{\gamma}(\mathbb{R}^n)$ consists of functions $f$ on $\mathbb{R}^n$ satisfying 
$$ \|f\|_{Lip_{\gamma}} := \sup_{\substack{x, y \in \mathbb{R}^n \\ x \neq y}} \frac{|f(x)-f(y)|}{|x-y|^{\gamma}} < \infty .$$
\end{definition}

The following result discusses the relationship between Lipschitz and Campanato spaces (p. 300--302 in \cite{Cuerva-Francia-Weightet-norm-inequalities-85}). 
\begin{theorem} \label{lem:Campanato=Lipschitz}
Given $0<\gamma < 1$ and $1 \leq r \leq \infty$, the spaces $Lip_{\gamma}$ and $\mathcal{L}_{\frac{\gamma}{n}, r, \floor*{\gamma}}$ coincide with equivalent norms. 
\end{theorem}

In view of Theorems \ref{thm:dual-Hp=Campanato} and \ref{lem:Campanato=Lipschitz}, we get that, whenever $\frac{n}{n+1} < p < 1$, the dual of $H^p$ coincides with $Lip_{n(\frac{1}{p}-1)}$ with equivalent norms. 

%%%%%%%%%%%%%%%%%%%%%%%%%%%%%%%%%%%
%%%%%%%%%%%%%%%%%%%%%%%%%%%%%%%%%%%

\subsection{Muckenhoupt \texorpdfstring{$A_p$}{}-weights} \label{subsec:Ap-weights} 
$A_p$-weights were introduced and studied by Muckenhoupt \cite{Muckenhoupt-TAMS1972}. Since then, weighted boundedness of several classical operators has continued to be a problem much sought after. Let us define these weights. 
\begin{definition}[$A_p$-weights] \label{def:Ap-weights}
Let $\omega \in L^1_{loc}(\mathbb{R}^n)$ be a non-negative function. Given $1<p<\infty$, we say that $w \in A_p$ if 
$$ [w]_{A_p(\mathbb{R}^n)} := \sup_{Q} \left( \frac{1}{|Q|} \int_Q \omega(x) \, dx \right) \left( \frac{1}{|Q|} \int_Q w(x)^{1-p^{\prime}} \, dx \right)^{p-1} < \infty, $$ 
where the supremum is taken over all cubes $Q$ and $\frac{1}{p} + \frac{1}{p^{\prime}} = 1$. 

For $p = 1 $, we say that $w \in A_1$ if there is a constant $C>0$ such that 
$$ M\omega(x) \leq C \, \omega(x) \quad \text{ for almost every } x \in \mathbb{R}^n, $$ 
and we also define the class $\displaystyle A_{\infty} = \bigcup_{1\leq p < \infty} A_p$. 
\end{definition}

The following lemma is an extension of Fefferman--Stein lemma \cite{Fefferman-Stein-Hp-spaces-Acta-Math-1972} in the context of weights, and can be found, for example as Lemma 2.1.3 on page 58 in \cite{Lu-Ding-Yan-book-singular-int-2007}. 
\begin{lemma} \label{lem:weighted-Fefferman-Stein}
Suppose $\omega \in A_{\infty}$ and there exists $0 < p_0 < \infty$ such that $Mf \in L^{p_0}(\omega)$. Then $$ \int_{\mathbb{R}^n} (Mf(x))^p \omega(x) \, dx \leq C \int_{\mathbb{R}^n} (M^{\#} f(x))^p\omega(x) \, dx $$
holds for $ p_0 \leq p < \infty$.
\end{lemma}

Recall also that, $|f| \leq Mf$ pointwise almost everywhere, and thus in view of Lemma \ref{lem:weighted-Fefferman-Stein} with $\omega = 1$, we have for $1 < p < \infty,$ 
\begin{align} \label{inequality:Sharp-maximal}
\| f \|_{L^p} \leq \| Mf \|_{L^p} \leq C \| M^{\#} f \|_{L^p} \ \ \text{for any} \ \ f \in L^p.
\end{align}

Let us record here another important result which will be useful later. It is known (see, Theorem 2.16 in Chapter IV on page 407 in \cite{Cuerva-Francia-Weightet-norm-inequalities-85}) that for $1<p<\infty$, 
\begin{align*}
(M\omega)^{1/p} \in A_1, \quad \text{for any non-negative function } \omega \in L^1_{loc}(\mathbb{R}^n).
\end{align*} 

Before we move on, let us also introduce the fractional maximal function $[f]_{\gamma , l}^{*}$, defined by 
\begin{align} 
[f]_{\gamma , l}^{*}(x) = \sup_{x\in Q} \left(\frac{1}{|Q|^{1-\frac{\gamma l}{n}}}\int_{Q}|f(y)|^l \, dy \right)^{\frac{1}{l}}. 
\end{align}
for $l \geq 1$ and $0< \gamma < n$, and the supremum is taken over all cubes $Q$.

In \cite{Chanillo-note-commutators-IUMJ-1982}, Chanillo proved the following boundedness result for $[f]_{\gamma , l}^{*}$. 
\begin{lemma}[\cite{Chanillo-note-commutators-IUMJ-1982}] \label{lem:fractional-maximal-Lp-Lq-bound}
For $1<l<p<\frac{n}{\gamma}$, $\frac{1}{q} = \frac{1}{p} - \frac{\gamma}{n}$, and $0< \gamma < n$, the fractional maximal function $[f]_{\gamma , l}^{*}$ is bounded from $L^p(\mathbb{R}^n)$ to $L^q(\mathbb{R}^n)$.
\end{lemma}

%%%%%%%%%%%%%%%%%%%%%%%%%%%%%%%%%%%
%%%%%%%%%%%%%%%%%%%%%%%%%%%%%%%%%%%

\subsection{Basic estimates} 
\label{subsec:basic-estimate-Omega-function} 

With $T_{\beta} = T_1 + T_2$, where $T_1$ is given by \eqref{def:operator-decomposition}, we start with recalling Theorem 2.1 of \cite{Chen-Guo-Extension-CZ-JFA-2021} about $L^q$-estimate of $T_1 f$, and Lemma 2.2 of \cite{Yu-Jiu-Li-CZ-JFA-2021} about $L^q$-estimate of $T_2 f$, which we shall need at several places in our analysis. 
\begin{theorem} [\cite{Chen-Guo-Extension-CZ-JFA-2021}]
\label{thm2.1:Chen-Guo-Extension-CZ-JFA-2021}
Let $\Omega$ satisfy condition \eqref{conditions:main-Omega-function} and $1<q<\infty$. There exists a constant $C> 0$ such that
\begin{align*} 
\left \|T_1 f\right\|_{{L^{q}}} \leq C \|f\|_{{L^{q}}}, 
\end{align*} 
holds true for all $f \in L^{q}(\mathbb{R}^{n})$ and $ 0 < \beta < n$.
\end{theorem}

\begin{lemma}[\cite{Yu-Jiu-Li-CZ-JFA-2021}] \label{lem2.2:Yu-Jiu-Li-CZ-JFA-2021} 
Let $\Omega \in L^\infty (S^{n-1})$ and $1 < q < \infty$. There exists a constant $C>0$ such that 
\begin{align*}
\left\| T_2 f \right\|_{L^q} \leq C \left\| \Omega \right\|_{L^\infty (S^{n-1})} \frac{\beta^{\frac{(q-1) n}{q}}}{\sqrt[q]{(n(q-1)-\beta q)}}\|f\|_{L^1}, 
\end{align*} 
holds true for all $f\in L^{1}(\mathbb{R}^{n})$ and $0 < \beta < \frac{(q-1) n}{q}$. 
\end{lemma}

Since the proof of Lemma \ref{lem2.2:Yu-Jiu-Li-CZ-JFA-2021} is elementary, we write the details here. 
\begin{proof}[Proof of Lemma \ref{lem2.2:Yu-Jiu-Li-CZ-JFA-2021} ] 
By definition of $T_2 f$, we have 
\begin{align*} 
\left| T_2 f(x) \right| &= \left| \int_{\mathbb{R}^n} \frac{\Omega(x-y)}{|x-y|^{n-\beta}} \left( 1 - \chi_\beta (|x-y|) \right) f(y) \, dy \right| \\ 
& \leq \left\| \Omega \right\|_{L^\infty (S^{n-1})} \int_{|x-y| \geq \frac{1}{\beta}} |x-y|^{\beta - n} \left| f(y) \right| dy, 
\end{align*} 
which implies that for $0 < \beta < \frac{(q-1) n}{q}$, 
\begin{align*} 
\left\| T_2 f \right\|_{L^q} & \leq \left\| \Omega \right\|_{L^\infty (S^{n-1})} \left( \int_{\mathbb{R}^n} \left| \int_{|x-y| \geq \frac{1}{\beta}} |x-y|^{\beta - n} \left| f(y) \right| dy \right|^q dx \right)^{1/q} \\
& \leq \left\| \Omega \right\|_{L^\infty (S^{n-1})} \int_{\mathbb{R}^n} \left| f(y) \right| \left( \int_{|x-y| \geq \frac{1}{\beta}} |x-y|^{(\beta - n)q} \, dx \right)^{1/q} dy \\ 
& = \left\| \Omega \right\|_{L^\infty (S^{n-1})} \frac{ \left( \frac{1}{\beta} \right)^{\beta / q} \beta^{\frac{(q-1) n}{q}}}{\sqrt[q]{(n(q-1)-\beta q)}}\|f\|_{L^1} \\ 
& \lesssim \left\| \Omega \right\|_{L^\infty (S^{n-1})} \frac{\beta^{\frac{(q-1) n}{q}}}{\sqrt[q]{(n(q-1)-\beta q)}}\|f\|_{L^1}, 
\end{align*} 
uniformly in $0 < \beta < \frac{(q-1) n}{q}$. 
\end{proof}

Also, we shall frequently make use of the following basic estimate whose proof is standard but we write the details for the sake of self-containment. 

\begin{lemma} \label{lem:Main-lemma}
Let $\Omega \in L^\infty (S^{n-1})$ be defined on $\mathbb{R}^n \setminus \{0\}$ as a homogeneous function of degree $0$. There exists a constant $C>0$ such that 
\begin{align*}
\left| \frac{\Omega(w-v)}{|w-v|^{n-\beta}} \chi_{\beta}(|w-v|) - \frac{\Omega(w)}{|w|^{n-\beta}} \chi_{\beta}(|w|) \right| \leq C \left( \frac{|v|}{|w|^{n+1}} \left\| \Omega \right\|_{L^\infty (S^{n-1})} + \frac{1}{|w|^n} \omega_{\infty} \left( 2 \frac{|v|}{|w|} \right) \right)
\end{align*} 
for all $0 < \beta < n$, and for every $|v|\leq r$ and $|w|\geq 2r.$ 
\end{lemma}
\begin{proof}
Note first that for $|v|\leq r$ and $|w|\geq 2r$, we have $\frac{|w|}{2} \leq |w-v| \leq \frac{3 |w|}{2}$, so that $|w-v|\sim |w|$. Now, 
\begin{align*}
& \left| \frac{\Omega(w-v)}{|w-v|^{n-\beta}}\chi_{\beta}(|w-v|)-\frac{\Omega(w)}{|w|^{n-\beta}}\chi_{\beta}(|w|) \right| \\
& \leq \left|\Omega(w-v)\chi_{\beta}(|w-v|)\left( \frac{1}{|w-v|^{n-\beta}}-\frac{1}{|w|^{n-\beta}} \right) \right|+\left| \frac{\chi_{\beta}(|w-v|)}{|w|^{n-\beta}}(\Omega(w-v)-\Omega(w)) \right| \\
& \quad + \left| \frac{\Omega(w)}{|w|^{n-\beta}}(\chi_{\beta}(|w-v|)-\chi_{\beta}(|w|)) \right| \\
& =: I_1 + I_2 + I_3.
\end{align*}

To estimate $I_1$, we make use of the fact that $supp (\chi_{\beta}) \subseteq \{v : |v|\leq \frac{2}{\beta}\}$, and the mean-value theorem to get that 
\begin{align*}
I_1 & = \left|\Omega(w-v)\chi_{\beta}(|w-v|)\left( \frac{1}{|w-v|^{n-\beta}}-\frac{1}{|w|^{n-\beta}} \right) \right| \\
& \lesssim \left\| \Omega \right\|_{L^\infty (S^{n-1})} \chi_{\beta}(|w-v|) \frac{|v|}{|w-v|^{n-\beta+1}} \\ 
& \lesssim \left\| \Omega \right\|_{L^\infty (S^{n-1})} \left(\frac{2}{\beta}\right)^\beta \frac{|v|}{|w|^{n+1}} \lesssim \left\| \Omega \right\|_{L^\infty (S^{n-1})} \frac{|v|}{|w|^{n+1}}.
\end{align*}

For $I_2$, note first that for $|w|\geq 2|v|$ we have $\left|\frac{w-v}{|w-v|}-\frac{w}{|w|} \right| \leq 2\frac{|v|}{|w|}$, which together with the fact that $\omega_\infty$ is non-decreasing in $\delta$ and the homogeneity of degree $0$ of $\Omega$ yields 
$$ |\Omega(w-v)-\Omega(w)| = \left| \Omega \left(\frac{w-v}{|w-v|} \right) - \Omega \left( \frac{w}{|w|} \right) \right| \leq \omega_{\infty} \left( \left| \frac{w-v}{|w-v|} - \frac{w}{|w|} \right| \right) \leq \omega_{\infty} \left( 2\frac{|v|}{|w|} \right). $$
Now, since $|w-v| \sim |w|$, utilizing as earlier the support of $\chi_{\beta}$, we get 
\begin{align*}
I_2 = \frac{\chi_{\beta}(|w-v|)}{|w|^{n-\beta}} \left| \Omega(w-v)-\Omega(w) \right| \lesssim \frac{1}{|w|^n} \omega_{\infty} \left(2\frac{|v|}{|w|} \right).
\end{align*}
Finally, we turn our attention to $I_3$. For the same, note from the definition of $\chi_\beta$ that 
\begin{align*}
|\chi_\beta(|w-v|)-\chi_\beta(|w|)|=0 & \quad \text{if} \ \ |w-v| \geq \frac{2}{\beta} \ \ \text{and} \ \ |w| \geq \frac{2}{\beta}; \\ 
& \quad \text{or} \ \ |w-v| \leq \frac{1}{\beta} \ \ \text{and} \ \ |w| \leq \frac{1}{\beta}. 
\end{align*} 
Equivalently, $ |w-v| \sim |w| \sim \frac{1}{\beta}$, or else $\chi_\beta(|x-y|)-\chi_\beta(|x|) = 0 $. 

Making use of the above observation, and  $|\chi_\beta^{\prime}(t)|\leq 2\beta$, we can apply the mean-value theorem to conclude that 
\begin{align} \label{est:chi-beta-difference}
|\chi_\beta(|w-v|) - \chi_\beta(|w|)| \lesssim \beta |v| \lesssim \frac{|v|}{|w|}, 
\end{align}  
which yields 
\begin{align*}
I_3 = \frac{\left| \Omega(w) \right| }{|w|^{n-\beta}} \left| \chi_{\beta}(|w-v|) - \chi_{\beta}(|w|) \right| \lesssim \left\| \Omega \right\|_{L^\infty (S^{n-1})} \frac{|v|}{|w|^{n+1}}. 
\end{align*}
This completes the proof of Lemma \ref{lem:Main-lemma}. 
\end{proof}

\begin{remark} \label{rem:lem:Main-lemma}
It follows easily from the proof of Lemma \ref{lem:Main-lemma} that given any $c > 1$, there exists a constant $\kappa_n > 0$ (depending on $c$) such that for any $|v|\leq r$ and $|w| \geq cr,$ the conclusion of Lemma \ref{lem:Main-lemma} holds true with $\omega_{\infty} \left( 2 \frac{|v|}{|w|} \right) $ replaced by $\omega_{\infty} \left( \kappa_n \frac{|v|}{|w|} \right) $.  
\end{remark}

%%%%%%%%%%%%%%%%%%%%%%%%%%%%%%%%%%%
%%%%%%%%%%%%%%%%%%%%%%%%%%%%%%%%%%%
%%%%%%%%%%%%%%%%%%%%%%%%%%%%%%%%%%%
%%%%%%%%%%%%%%%%%%%%%%%%%%%%%%%%%%%
%%%%%%%%%%%%%%%%%%%%%%%%%%%%%%%%%%%
%%%%%%%%%%%%%%%%%%%%%%%%%%%%%%%%%%%
%%%%%%%%%%%%%%%%%%%%%%%%%%%%%%%%%%%
%%%%%%%%%%%%%%%%%%%%%%%%%%%%%%%%%%%
%%%%%%%%%%%%%%%%%%%%%%%%%%%%%%%%%%%
%%%%%%%%%%%%%%%%%%%%%%%%%%%%%%%%%%%
%%%%%%%%%%%%%%%%%%%%%%%%%%%%%%%%%%%

\section{Boundedness of CZ type operators} \label{sec:proofs-FCZ-Lipschitz-spaces}
We start with proving Theorem \ref{thm:CZ-Lipschitz-space} in Subsection \ref{subsec:CZ-Lipschitz}. Later on, we shall establish a variant of Hardy space estimates and some weighted $L^p$-estimates of $T_\beta$ in Subsections \ref{subsec:Hardy-CZ} and \ref{subsec:Weighted-CZ}. 

%%%%%%%%%%%%%%%%%%%%%%%%%%%%%%%%%%%
%%%%%%%%%%%%%%%%%%%%%%%%%%%%%%%%%%%

\subsection{Boundedness on Lipschitz spaces: Proof of Theorem \ref{thm:CZ-Lipschitz-space}} \label{subsec:CZ-Lipschitz} 
We prove Theorem \ref{thm:CZ-Lipschitz-space} in this subsection. Note that with conditions on indices $p$ and $q$ as in Theorem \ref{thm:CZ-Lipschitz-space}, we have $\frac{(q-1) n}{q} = \frac{(1-p) n}{p} < 1$. Therefore, in this subsection, unless otherwise mentioned, we shall always be concerned with estimates with the restriction on $\beta$ to be that it is less than $1$. 

\begin{proof}[Proof of Theorem \ref{thm:CZ-Lipschitz-space}]
It suffices to show that 
$$ \|T_{1}f\|_{Lip_{n(\frac{1}{p}-1)}} \lesssim \|f\|_{Lip_{n(\frac{1}{p}-1)}} \quad \text{and} \quad  \|T_{2}f\|_{Lip_{n(\frac{1}{p}-1)}} \lesssim \frac{\beta^{\frac{(q-1) n}{q}}}{\sqrt[q]{(n(q-1)-\beta q)}} \|f\|_{BMO},$$ 
uniformly in $0 < \beta < \frac{(q-1) n}{q}$. 

But, one can work with the adjoint of operators $T_1$ and $T_2$, and making use of the duality of $H^p$-spaces, it is equivalent to prove that 
\begin{align} \label{main-est:proof-thm:CZ-Lipschitz-space}
\|T_{1}f\|_{H^{p}} \lesssim \|f\|_{H^{p}} \quad \text{and} \quad \|T_{2}f\|_{H^{1}} \lesssim \frac{\beta^{\frac{(q-1) n}{q}}}{\sqrt[q]{(n(q-1)-\beta q)}} \|f\|_{H^{p}}, 
\end{align} 
uniformly in $0 < \beta < \frac{(q-1) n}{q}$. 

\medskip \noindent \textbf{\underline{Analysis of operator $T_1$ in estimate \eqref{main-est:proof-thm:CZ-Lipschitz-space}}:} $H^p$ boundedness of $T_1$ can be proved similarly as that of Calder\'{o}n-Zygmund operators. But, for the sake of completeness, we write the proof here. We will closely follow \cite{Ding-Lu-hom-frac-Hardy-Tohoku-2000}. 

As $p>\frac{n}{n+\alpha}$, we can choose $\varepsilon > 0$ such that $\frac{1}{p}-1<\varepsilon<\frac{\alpha}{n}$. Fix $l \in (1, \infty)$, and write $a_{0} = 1 - \frac{1}{p} + \varepsilon$ and $b_{0} = 1 - \frac{1}{l} + \varepsilon$. Also, since we are working with $p > \frac{n}{n+\alpha} \geq \frac{n}{n+1}$, we have $\floor*{n(\frac{1}{p}-1)} = 0,$ and therefore we can make use of Lemma \ref{lem:Hp1-Hp2-via-molecules} with $s_1 = s_2 = 0$. 

Let $a(x)$ be a $(p,l,0)$-atom supported on $B=B(0,d)$. In view of Lemma \ref{lem:Hp1-Hp2-via-molecules}, it suffices to show that $T_{1}a$ is a $(p,l,0,\varepsilon)$-molecule, satisfying the following conditions: 
\begin{enumerate}[label=(\roman*), start=1]
\item $\displaystyle \mathcal{N}_l(T_1a) := \left\| T_{1}a \right\|_{L^l}^{\frac{a_{0}}{b_{0}}} \ \left\| |\cdot |^{nb_{0}}(T_{1}a)(\cdot) \right\|_{L^l}^{1-\frac{a_{0}}{b_{0}}} \leq C \ ( \text{independent of} \ a)$,  \label{part(i)-T_1-proof-thm:CZ-Lipschitz-space}

\item $\displaystyle \int(T_{1}a)(x) \, dx =0.$ \label{part(ii)-T_1-proof-thm:CZ-Lipschitz-space}
\end{enumerate}

\medskip \textbf{Proof of part \ref{part(i)-T_1-proof-thm:CZ-Lipschitz-space}:} Note that it follows from Theorem \ref{thm2.1:Chen-Guo-Extension-CZ-JFA-2021} that $\|T_{1}a\|_{L^l} \leq C \|a\|_{L^l}$. Now, 
\begin{align*}
\left\| |\cdot |^{nb_{0}}(T_{1}a)(\cdot) \right\|_{L^l}  \leq \left\| |\cdot |^{nb_{0}}(T_{1}a)(\cdot) {\mathbbm{1}}_{2B}(\cdot) \right\|_{L^l} + \left\| |\cdot |^{nb_{0}}(T_{1}a)(\cdot) {\mathbbm{1}}_{(2B)^c}(\cdot) \right\|_{L^l} & := A_{1}+A_{2}.
\end{align*}
Again, using Theorem \ref{thm2.1:Chen-Guo-Extension-CZ-JFA-2021}, we get
\begin{align*}
A_{1} = \left( \int_{2B}|x|^{nb_{0}l}|(T_{1}a)(x)|^{l} \, dx \right)^{\frac{1}{l}} \lesssim |B|^{b_{0}} \|T_{1}a\|_{L^l} \lesssim |B|^{b_{0}} \|a\|_{L^l}.
\end{align*}
Next, using Minkowski's inequality, Lemma \ref{lem:Main-lemma} and the vanishing condition of atom $a$, 
\begin{align*} 
A_{2} & = \left(\int_{({2B})^{c}} \left|\int_{B}\frac{\Omega(x-y)}{|x-y|^{n-\beta}}\chi_\beta (|x-y|)a(y) \, dy \right|^l |x|^{nb_{0}l} \, dx \right)^\frac{1}{l}\\
& =  \left(\int_{({2B})^{c}} \left|\int_{B}\left(\frac{\Omega(x-y)}{|x-y|^{n-\beta}}\chi_\beta (|x-y|)-\frac{\Omega(x)}{|x|^{n-\beta}}\chi_\beta (|x|)\right)a(y) \, dy \right|^l |x|^{nb_{0}l} \, dx \right)^\frac{1}{l}\\
& \leq \int_{B} |a(y)| \left(\int_{({2B})^{c}}\left|\frac{\Omega(x-y)}{|x-y|^{n-\beta}}\chi_\beta (|x-y|)-\frac{\Omega(x)}{|x|^{n-\beta}}\chi_\beta (|x|)\right|^l |x|^{nb_{0}l} \, dx \right)^\frac{1}{l} dy \\
& \lesssim \int_{B} |a(y)| \left(\int_{({2B})^{c}} \left| \frac{|y|}{|x|^{(n+1)}} \left\| \Omega \right\|_{L^\infty (S^{n-1})} + \frac{1}{|x|^{n}} \omega_{\infty} \left(2\frac{|y|}{|x|} \right) \right|^{l} |x|^{nb_{0}l} \, dx \right)^\frac{1}{l} dy \\
& \lesssim \left( \int_{B} |a(y)| |y| \, dy \right) \left(\int_{({2B})^{c}} |x|^{- (n+1-nb_{0})l} \, dx \right)^\frac{1}{l} \\ 
& \quad +  \int_{B} |a(y)| \left(\int_{({2B})^{c}} |x|^{- (n-nb_{0}) l} \omega_{\infty}^{l} \left(2\frac{|y|}{|x|}\right) dx \right)^\frac{1}{l} \, dy \\ 
& =: E_1 + E_2.
\end{align*} 
By choice of $\varepsilon$, we have $(-n-1+n b_0)l+n < 0$, therefore using H\"older's inequality, we get 
\begin{align*}
E_1 \lesssim \|a\|_{L^l} \left(\int_{B}|y|^{l^\prime} \, dy \right)^{\frac{1}{l^\prime}} d^{-n-1+nb_0+\frac{n}{l}} \lesssim \|a\|_{L^l} \ d^{1+n-\frac{n}{l}} d^{-n-1+nb_0+\frac{n}{l}} \lesssim |B|^{b_0}\|a\|_{L^l}. 
\end{align*}
On the other hand, by choice of $\varepsilon$, Dini condition \eqref{conditions:alpha-Dini-Omega-function} on $\Omega$, and the H\"older's inequality, we get 
\begin{align*}
E_2 & = \int_{B} |a(y)| \left(\int_{({2B})^{c}} |x|^{-(n-nb_0)l} \omega_{\infty}^{l} \left( 2\frac{|y|}{|x|} \right) dx \right)^\frac{1}{l} dy \\
& \lesssim \|a\|_{L^l}|B|^{\frac{1}{l^\prime}} \sum_{j=1}^{\infty} \left( \int_{2^jd\leq |x| \leq 2^{j+1}d} |x|^{-(n-nb_0)l} \, dx \right)^\frac{1}{l} \omega_{\infty} \left( 2^{-j+1} \right) \\
&  \lesssim \|a\|_{L^l}|B|^{\frac{1}{l^\prime}} \sum_{j=0}^{\infty} (2^jd)^{n\varepsilon} \omega_{\infty} \left(2^{-j} \right) \\
& \lesssim \|a\|_{L^l}d^{nb_0} \sum_{j=0}^{\infty}(2^j)^{\alpha} \omega_{\infty} \left(2^{-j} \right) \\
& \lesssim \|a\|_{L^l}|B|^{b_0} \left( \omega_{\infty}(1) + \sum_{j=1}^{\infty} \int_{2^{-j}}^{2^{-j+1}} \frac{\omega_{\infty}(t)}{t^{1+\alpha}} \, dt \right) \\
& \lesssim \|a\|_{L^l}|B|^{b_0} \left( \omega_{\infty}(1) + \int_{0}^{1} \frac{\omega_{\infty}(t)}{t^{1+\alpha}} \, dt \right) \\
& \lesssim |B|^{b_0}\|a\|_{L^l}.
\end{align*} 

Combining estimates of $E_1$ and $E_2$, we get $\displaystyle \left\| |\cdot |^{nb_{0}}(T_{1}a)(\cdot) \right\|_{L^l} \lesssim |B|^{b_0}\|a\|_{L^l}$, and therefore 
\begin{align*}
\mathcal{N}_l(T_1a) = \left\| T_{1}a \right\|_{L^l}^{\frac{a_{0}}{b_{0}}} \ \left\| |\cdot |^{nb_{0}}(T_{1}a)(\cdot) \right\|_{L^l}^{1-\frac{a_{0}}{b_{0}}} & \lesssim \|a\|_{L^l}^{\frac{a_{0}}{b_{0}}} \ |B|^{b_0-a_0} \ \|a\|_{L^l}^{1-\frac{a_{0}}{b_{0}}} = |B|^{b_0-a_0}\|a\|_{L^l} \leq 1, 
\end{align*} 
completing the proof of part \ref{part(i)-T_1-proof-thm:CZ-Lipschitz-space}.

\medskip \textbf{Proof of part \ref{part(ii)-T_1-proof-thm:CZ-Lipschitz-space}:} Following the calculations done towards the ending of page 160 in \cite{Ding-Lu-hom-frac-Hardy-Tohoku-2000}, we know that proving $\int T_{1}a(x) dx =0$ is equivalent to showing that $ (T_1a)^{\widehat{}}(\xi) \to 0$ as $\xi \to 0$. Towards that end, repeating calculations from page 160 of  \cite{Ding-Lu-hom-frac-Hardy-Tohoku-2000}, we get that 
$$ \left|\left(\frac{\Omega(\cdot)}{|\cdot|^{n-\beta}}\chi_{\beta}(|\cdot|) \right)^{\widehat{}}(\xi)\right| \leq C + \sum_{j=1}^{\infty}\left|\widehat{K_j}(\xi) \right|, $$ 
where 
$$ K_j(x) = \frac{\Omega(x)}{|x|^{n-\beta}} \chi_{\beta}(|x|) \, {\mathbbm{1}}_{[2^{j-1},2^j]}(|x|). $$
Taking advantage of the properties of function $\chi_{\beta}$, including its support condition, one can essentially repeat the proof of Lemma 5 of \cite{Ding-Lu-hom-frac-Hardy-Tohoku-2000}, modulo some obvious modifications, to conclude that there exists $0<\sigma<1$ such that 
$$ \left|\widehat{K_j}(\xi) \right| \leq C 2^{- \sigma j / 2} |\xi|^{-\sigma / 2}. $$

With the above estimate at hand, rest of the proof follows from the arguments given right after Lemma 5 in \cite{Ding-Lu-hom-frac-Hardy-Tohoku-2000}.

\medskip \noindent \textbf{\underline{Analysis of operator $T_2$ in estimate \eqref{main-est:proof-thm:CZ-Lipschitz-space}}:} 
With $p>\frac{n}{n+\alpha}$, this time we choose $\varepsilon > 0$ such that $0<\varepsilon<1-\frac{1}{p}+\frac{\alpha}{n}$. Fix $q \in (1, \infty)$ such that $\frac{1}{p}+\frac{1}{q}=2$, and write $a_{0} = \varepsilon$ and $b_{0} = 1 - \frac{1}{q} + \varepsilon$. Clearly, $n b_0 < \alpha \leq 1$, and we shall later make use of it. As observed in the analysis of operator $T_1$, we can again make use of Lemma \ref{lem:Hp1-Hp2-via-molecules} with $s_1 = s_2 = 0$. 

Let $a(x)$ be a $(p,q,0)$-atom supported on $B=B(0,d)$. In view of Lemma \ref{lem:Hp1-Hp2-via-molecules}, it suffices to show that $T_{2} a$ is a $(1,q,0,\varepsilon)$ molecule, satisfying the following conditions: 
\begin{enumerate}[label=(\roman*), start=3]
\item $\displaystyle \mathcal{N}_q(T_2a) := \left\| T_{2}a \right\|_{L^q}^{\frac{a_{0}}{b_{0}}} \ \left\| |\cdot|^{nb_{0}}(T_{2}a)(\cdot) \right\|_{L^q}^{1-\frac{a_{0}}{b_{0}}} \leq C \ ( \text{independent of} \ a)$, \label{part(iii)-T_2-proof-thm:CZ-Lipschitz-space}

\item $\displaystyle \int(T_{2}a)(x) \, dx =0.$ \label{part(iv)-T_2-proof-thm:CZ-Lipschitz-space}
\end{enumerate}

\medskip \textbf{Proof of part \ref{part(iii)-T_2-proof-thm:CZ-Lipschitz-space}:} In order to study $\mathcal{N}_q(T_2a)$, first note that $\|T_{2}a\|_{L^q}$ can be estimated by Lemma \ref{lem2.2:Yu-Jiu-Li-CZ-JFA-2021}, and we get that for any $0 < \beta < \frac{(q-1) n}{q}$, 
$$ \| T_{2}a \|_{L^q} \lesssim \frac{\beta^{\frac{(q-1) n}{q}}}{\sqrt[q]{(n(q-1)-\beta q)}} \|a\|_{L^1}.$$ 
Now, 
\begin{align*}
\left\| |\cdot |^{nb_{0}}(T_{2}a)(\cdot) \right\|_{L^q} \leq \left\| |\cdot |^{nb_{0}}(T_{2}a)(\cdot) {\mathbbm{1}}_{2B}(\cdot) \right\|_{L^q} + \left\| |\cdot |^{nb_{0}}(T_{2}a)(\cdot) {\mathbbm{1}}_{(2B)^c}(\cdot) \right\|_{L^q} =: B_1 + B_2.
\end{align*}
We can estimate $B_1$ too by Lemma \ref{lem2.2:Yu-Jiu-Li-CZ-JFA-2021}. In fact, for any $0 < \beta < \frac{(q-1) n}{q}$, 
\begin{align*}
B_{1} = \left( \int_{2B}|x|^{nb_{0}q}|(T_{2}a)(x)|^{q} \, dx  \right)^{\frac{1}{q}} \lesssim |B|^{b_{0}} \|T_{2}a\|_{L^q} \lesssim \frac{\beta^{\frac{(q-1) n}{q}}}{\sqrt[q]{(n(q-1)-\beta q)}} |B|^{b_{0}} \|a\|_{L^1}.
\end{align*}

Next, using Minkowski's inequality and the vanishing condition of atom $a$, term $B_2$ can be dominated by the sum of the following three terms: 
\begin{align*}
F_1 &= \int_{B}|a(y)| \left( \int_{(2B)^c} |\Omega(x-y)|^q\left|  \frac{1}{|x-y|^{n-\beta}}- \frac{1}{|x|^{n-\beta}} \right|^{q}(1-\chi_\beta (|x-y|))^q|x|^{nb_{0}q} \, dx \right)^{\frac{1}{q}} dy\\
F_2 &= \int_{B}|a(y)| \left( \int_{(2B)^c}\frac{1}{|x|^{(n-\beta)q}} \left| \Omega(x-y)-\Omega(x) \right|^{q}(1-\chi_\beta (|x-y|))^q|x|^{nb_{0}q} \, dx \right)^{\frac{1}{q}} dy\\
F_3 &= \int_{B}|a(y)| \left( \int_{(2B)^c}\frac{|\Omega(x)|^q}{|x|^{(n-\beta)q}} \left|  \chi_\beta (|x-y|)- \chi_\beta(|x|) \right|^{q}|x|^{nb_{0}q} \, dx \right)^{\frac{1}{q}} dy.
\end{align*}

As earlier, when $x \in (2B)^c$ and $y \in B$ we have $|x-y|\sim |x|$, and by mean-value theorem
$$\left|  \frac{1}{|x-y|^{n-\beta}}- \frac{1}{|x|^{n-\beta}} \right|\lesssim \frac{|y|}{|x|^{(n-\beta+1)}}.$$

With $nb_0 <\alpha \leq 1$ and $0<\beta< \frac{(q-1) n}{q}$, we are then lead to the following estimate of $F_1$. 
\begin{align*}
F_{1} & = \int_{B}|a(y)| \left( \int_{(2B)^c} |\Omega(x-y)|^q \left| \frac{1}{|x-y|^{n-\beta}}- \frac{1}{|x|^{n-\beta}} \right|^{q}(1-\chi_\beta (|x-y|))^q|x|^{nb_{0}q} \, dx \right)^{\frac{1}{q}} dy\\
& \lesssim \int_{B}|a(y)| \sum_{j=1}^{\infty}\left( \int_{2^jd\leq |x| \leq 2^{j+1}d} \frac{|y|^q}{|x|^{(n-\beta+1)q}}(1-\chi_\beta (|x-y|))^q|x|^{nb_{0}q} \, dx \right)^{\frac{1}{q}} dy \\
& \lesssim |B|^{b_0}\int_{B}|a(y)| \sum_{j=1}^{\infty}\frac{1}{2^{(1-nb_0)j}}\left( \int_{2^jd\leq |x| \leq 2^{j+1}d} \frac{1}{|x|^{(n-\beta)q}}(1-\chi_\beta (|x-y|))^q  \, dx \right)^{\frac{1}{q}} dy \\
& \lesssim |B|^{b_0}\int_{B}|a(y)| \sum_{j=1}^{\infty}\frac{1}{2^{(1-nb_0)j}}\left( \int_{|x-y|\geq \frac{1}{\beta}} \frac{1}{|x-y|^{(n-\beta)q}}  \, dx \right)^{\frac{1}{q}} dy \\
& \lesssim \frac{\beta^{\frac{(q-1) n}{q}}}{\sqrt[q]{(n(q-1)-\beta q)}} |B|^{b_0} \|a\|_{L^1}. 
\end{align*}

Next, similar to our analysis of $E_2$, we can perform the following estimation for $F_2$. 
\begin{align*}
F_2 & = \int_{B}|a(y)| \left( \int_{(2B)^c}\frac{1}{|x|^{(n-\beta)q}} \left| \Omega(x-y)-\Omega(x) \right|^{q} (1-\chi_\beta (|x-y|))^q|x|^{nb_{0}q} \, dx \right)^{\frac{1}{q}} dy \\
& \lesssim |B|^{b_0} \int_{B} |a(y)| \sum_{j=1}^{\infty}(2^j )^{n b_0}\left( \int_{2^jd\leq |x| \leq 2^{j+1}d}\frac{ \left| \Omega(x-y)-\Omega(x) \right|^{q}}{|x|^{(n-\beta)q}}(1-\chi_\beta (|x-y|))^q   \, dx \right)^\frac{1}{q} dy \\
& \lesssim |B|^{b_0} \int_{B} |a(y)| \sum_{j=1}^{\infty}(2^j )^{n b_0} \omega_{\infty} \left(2^{-j} \right) \left( \int_{|x-y| \geq \frac{1}{\beta}}\frac{1}{|x-y|^{(n-\beta)q}}  \, dx \right)^\frac{1}{q} dy \\
& \lesssim \frac{\beta^{\frac{(q-1) n}{q}}}{\sqrt[q]{(n(q-1)-\beta q)}} |B|^{b_0}\|a\|_{L^1} \sum_{j=1}^{\infty}(2^j)^{\alpha} \omega_{\infty} \left(2^{-j} \right) \\
& \lesssim \frac{\beta^{\frac{(q-1) n}{q}}}{\sqrt[q]{(n(q-1)-\beta q)}} |B|^{b_0}\|a\|_{L^1}.
\end{align*}

In analysing $F_3$, we make use of estimate \eqref{est:chi-beta-difference}, which leads us to 
\begin{align*}
F_3 & = \int_{B}|a(y)| \left( \int_{(2B)^c}\frac{|\Omega(x)|^q}{|x|^{(n-\beta)q}} \left|  \chi_\beta (|x-y|)- \chi_\beta(|x|) \right|^{q}|x|^{nb_{0}q} \, dx \right)^{\frac{1}{q}} dy \\ 
& \lesssim |B|^{b_0} \int_{B}|a(y)|\sum_{j=1}^{\infty}(2^j)^{n b_0} \left( \int_{2^j d \leq |x| < 2^{j+1}d} \frac{1}{|x|^{(n-\beta)q}} \frac{|y|^q}{|x|^q} \, \chi_{\{|x| \sim \frac{1}{\beta}\}} (x) \, dx \right)^{\frac{1}{q}} dy \\
& \lesssim |B|^{b_0} \int_{B}|a(y)|\sum_{j=1}^{\infty}\frac{1}{2^{(1-n b_0)j}} \left( \int_{2^j d \leq |x| < 2^{j+1}d} \frac{1}{|x|^{(n-\beta)q}} \, \chi_{\{|x| \sim \frac{1}{\beta}\}} (x) \, dx \right)^{\frac{1}{q}} dy\\
& \lesssim \frac{\beta^{\frac{(q-1) n}{q}}}{\sqrt[q]{(n(q-1)-\beta q)}} |B|^{b_0}\|a\|_{L^1}. 
\end{align*}
 
Combining all the above estimates and using the relation $\frac{1}{p}+\frac{1}{q}=2$, we have
\begin{align*}
\mathcal{N}_q(T_2a) & = \|T_{2}a\|_{L^q}^{\frac{a_{0}}{b_{0}}} \| |\cdot |^{nb_{0}}(T_{2}a)(\cdot)\|_{L^q}^{1-\frac{a_{0}}{b_{0}}} \\ 
& \lesssim \left(\frac{\beta^{\frac{(q-1) n}{q}}}{\sqrt[q]{(n(q-1)-\beta q)}}\right)^{\frac{a_{0}}{b_{0}}} \|a\|_{L^1}^{\frac{a_{0}}{b_{0}}} \left(\frac{\beta^{\frac{(q-1) n}{q}}}{\sqrt[q]{(n(q-1)-\beta q)}}\right)^{1-{\frac{a_{0}}{b_{0}}}} |B|^{b_0-a_0}\|a\|_{L^1}^{1-\frac{a_{0}}{b_{0}}} \\
& \lesssim \frac{\beta^{\frac{(q-1) n}{q}}}{\sqrt[q]{(n(q-1)-\beta q)}}|B|^{b_0-a_0}\|a\|_{L^1} \\
& \lesssim \frac{\beta^{\frac{(q-1) n}{q}}}{\sqrt[q]{(n(q-1)-\beta q)}} |B|^{1-\frac{1}{q}}|B|^{1-\frac{1}{p}} \\
& \leq \frac{\beta^{\frac{(q-1) n}{q}}}{\sqrt[q]{(n(q-1)-\beta q)}}.
\end{align*}

\medskip \textbf{Proof of part \ref{part(iv)-T_2-proof-thm:CZ-Lipschitz-space}:} The proof of $\displaystyle \int T_{2}a(x) \, dx =0$ can be done similar to the one for $\displaystyle \int T_{1}a(x) \, dx = 0$ as in part \ref{part(ii)-T_1-proof-thm:CZ-Lipschitz-space}, and the same would be accomplished if we could show that 
\begin{align} \label{est:proof-main-thorem-technical-est-T2} 
& \left| \int_{\mathbb{R}^n} \frac{\Omega(x)}{|x|^{n-\beta}}(1-\chi_{\beta}(|x|))e^{-2\pi \iota \xi \cdot x} \, dx \right| \lesssim_{\beta, \eta} (1 + |\xi|^{-\eta}),
\end{align}
for all $0 < \beta < \eta < 1$ and for every $\xi \in \mathbb{R}^n$. 

Note that, we can write 
\begin{align*}
& \left| \int_{\mathbb{R}^n} \frac{\Omega(x)}{|x|^{n-\beta}} (1-\chi_{\beta}(|x|)) e^{-2\pi \iota \xi \cdot x} \, dx \right| \\ 
& \leq \left| \int_{\frac{1}{\beta} < |x| \leq \frac{2}{\beta}} \frac{\Omega(x)}{|x|^{n-\beta}} (1-\chi_{\beta}(|x|)) e^{-2\pi \iota \xi \cdot x} \, dx \right| + \left| \int_{|x| > \frac{2}{\beta}} \frac{\Omega(x)}{|x|^{n-\beta}} e^{-2\pi \iota \xi \cdot x} \, dx \right| = : I + II. 
\end{align*}
Estimating $I$ is straightforward. In fact, uniformly in $\beta \in (0, 1)$, 
\begin{align*}
I & \leq \int_{\frac{1}{\beta} < |x| \leq \frac{2}{\beta}} \frac{\left| \Omega(x) \right|}{|x|^{n-\beta}} \, dx =  \int_{S^{n-1}} \left| \Omega(x') \right| \, d \sigma (x') \int_{\frac{1}{\beta}}^{\frac{2}{\beta}} r^{\beta-1} \, dr \lesssim \frac{2^\beta - 1}{\beta} \beta^{-\beta} \lesssim 1. 
\end{align*}
On the other hand, 
\begin{align*}
II & = \left|\sum_{k=1}^{\infty}\int_{\frac{1}{\beta}2^k\leq |x| < \frac{1}{\beta}2^{k+1}} \frac{\Omega(x)}{|x|^{n-\beta}} e^{-2\pi \iota \xi \cdot x} \, dx \right| \\
& \leq \sum_{k=1}^{\infty} \left| \int_{S^{n-1}}\Omega(x^{\prime}) \left( \int_{\frac{1}{\beta}2^k}^{\frac{1}{\beta}2^{k+1}} e^{-2\pi \iota r \xi \cdot x^{\prime}} r^{\beta - 1} dr \right) d\sigma(x^{\prime}) \right| \\ 
& \lesssim_{\eta} \sum_{k=1}^{\infty} \left( \frac{2^k}{\beta}\right)^{\beta} \left| \frac{2^k}{\beta} \xi\right|^{-\eta}, 
\end{align*}
for any $0 < \eta < 1$. Here, the last estimate follows from inequality (2.9) of \cite{Chen-Guo-Extension-CZ-JFA-2021}. 

Then, with an $\eta \in (\beta, 1)$ fixed, we would get 
$$ II \lesssim_{\eta} |\xi|^{-\eta} \sum_{k=1}^{\infty} 2^{k(\beta - \eta)} \lesssim_{\eta} |\xi|^{-\eta}.$$ 

Combining estimates for $I$ and $II$ together imply the claimed inequality \eqref{est:proof-main-thorem-technical-est-T2}. With that the proof of part \ref{part(iv)-T_2-proof-thm:CZ-Lipschitz-space} can be completed following the idea mentioned at the end of the proof of part \ref{part(ii)-T_1-proof-thm:CZ-Lipschitz-space}, and this completes the proof of Theorem \ref{thm:CZ-Lipschitz-space}.
\end{proof} 

%%%%%%%%%%%%%%%%%%%%%%%%%%%%%%%%%%%
%%%%%%%%%%%%%%%%%%%%%%%%%%%%%%%%%%%
%%%%%%%%%%%%%%%%%%%%%%%%%%%%%%%%%%%
%%%%%%%%%%%%%%%%%%%%%%%%%%%%%%%%%%%

\subsection{Boundedness on Hardy spaces} \label{subsec:Hardy-CZ}
We have the following variant of Theorem \ref{thm:CZ-Lipschitz-space} concerning estimates of operators $T_\beta$ on Hardy spaces $H^p$, with error term taken in other Hardy spaces.
\begin{theorem} \label{thm:CZ-Hardy-estimate}
Let $\Omega$ satisfy conditions \eqref{conditions:main-Omega-function} and \eqref{conditions:alpha-Dini-Omega-function} for some $0< \alpha \leq 1$. Given $\frac{n}{n+\alpha}<m<p \leq 1$, let $q>1$ be such that $\frac{1}{m} - \frac{1}{p} = 1 - \frac{1}{q}$. Then there exists a constant $ C>0$ such that 
\begin{align*}
\left\|T_{\beta} f\right\|_{H^p} \leq C \left( \|f\|_{H^p}+\frac{\beta^{\frac{(q-1) n}{q}}}{\sqrt[q]{(n(q-1)-\beta q)}}\|f\|_{H^{m}} \right).
\end{align*}
holds true for every $0 < \beta < \frac{(q-1) n}{q}$ and for any $f \in H^p\cap H^{m} $. 
\end{theorem}
\begin{proof}
As earlier, it suffices to show that 
$$\|T_{1}f\|_{H^p}\leq C \|f\|_{H^p} \quad \text{and} \quad \|T_{2}f\|_{H^p}\leq C \frac{\beta^{\frac{(q-1) n}{q}}}{\sqrt[q]{(n(q-1)-\beta q)}} \|f\|_{H^{m}}.$$
Estimates for $T_1f$ were already established in the proof of Theorem \ref{thm:CZ-Lipschitz-space}, so we are only left to study $T_2 f$. As $m>\frac{n}{n+\alpha}$, we can choose $\varepsilon$ such that $\frac{1}{p}-1<\varepsilon<\frac{1}{q}-1+\frac{\alpha}{n}$. Write $a_{0}=1-\frac{1}{p}+\varepsilon$ and  $b_{0}=1-\frac{1}{q}+\varepsilon$. Let $a(x)$ be an $(m,q,0)$-atom supported on $B=B(0,d)$. It suffices to show that $T_{2}a$ is $(p,q,0,\varepsilon)$-molecule, satisfying the following conditions: 
\begin{enumerate}[label=(\roman*), start=5]
\item $\displaystyle \mathcal{N}_q(T_2a) = \left\|T_{2}a \right\|_{L^q}^{\frac{a_{0}}{b_{0}}} \left\| |\cdot |^{nb_{0}}(T_{2}a)(\cdot) \right\|_{L^q}^{1-\frac{a_{0}}{b_{0}}} \leq C \ ( \text{independent of}\  a).$ \label{part(v)-T_2-proof-thm:CZ-Hardy-estimate} 

\item $\displaystyle \int (T_{2}a)(x)  \, dx =0.$ \label{part(vi)-T_2-proof-thm:CZ-Hardy-estimate}
\end{enumerate}

Both \ref{part(v)-T_2-proof-thm:CZ-Hardy-estimate}  and \ref{part(vi)-T_2-proof-thm:CZ-Hardy-estimate} can be proved in an exactly similar manner to the ideas used in the proof of Theorem \ref{thm:CZ-Lipschitz-space}.
Here again we make use of $nb_0-1\leq nb_0-\alpha<0$ as $\varepsilon < \frac{1}{q}-1+\frac{\alpha}{n}$, so that 
\begin{align*}
\mathcal{N}_q(T_2a) = \left\| T_{2}a \right\|_{L^q}^{\frac{a_{0}}{b_{0}}} \left\| |\cdot |^{nb_{0}}(T_{2}a)(\cdot) \right\|_{L^q}^{1-\frac{a_{0}}{b_{0}}} & \lesssim \frac{\beta^{\frac{(q-1) n}{q}}}{\sqrt[q]{(n(q-1)-\beta q)}}|B|^{b_0-a_0}\|a\|_{L^1} \\
& \lesssim \frac{\beta^{\frac{(q-1) n}{q}}}{\sqrt[q]{(n(q-1)-\beta q)}} |B|^{\frac{1}{p}-\frac{1}{q}+1-\frac{1}{m}} \\
& \lesssim \frac{\beta^{\frac{(q-1) n}{q}}}{\sqrt[q]{(n(q-1)-\beta q)}},
\end{align*}
where we have used the fact that $0< \beta < \frac{(q-1) n}{q}$ and $\frac{1}{m}-\frac{1}{p}=1-\frac{1}{q}$. 

This completes the proof of Theorem \ref{thm:CZ-Hardy-estimate}.
\end{proof}

\subsection{Weighted boundedness} \label{subsec:Weighted-CZ}

This subsection is devoted to an extension of Theorem \ref{thm:Yu-Jiu-Li-CZ-JFA-2021} to weighted $L^p$-spaces. For results in the case of $\beta = 0$, we refer to \cite{Lu-Ding-Yan-book-singular-int-2007}. Following is our main result in this context. 
\begin{theorem} \label{thm:CZ-weighted-Lp-estimate}
Given $1<p<\infty$ and $\omega \in A_p(\mathbb{R}^n)$, there exists $q \in (1, p)$ such that for every $\Omega$ satisfying conditions \eqref{conditions:main-Omega-function} and \eqref{conditions:Dini-Omega-function}, there exists a constant $C > 0$ such that
\begin{align*}
\left\|T_{\beta} f\right\|_{L^{p}(\omega)} \leq C \left(\|f\|_{L^{p}(\omega)} + \frac{\beta^{n(1-\frac{1}{p})}}{n(1 - \frac{q}{p}) - \beta} \ \|f\|_{L^{1} ((M\omega)^{\frac{1}{p}})} \right)
\end{align*}
holds true for every $ 0 < \beta < \frac{(p-q) n}{p}$ and for any $f\in L^{p}(\omega) \cap L^1 ((M\omega)^{\frac{1}{p}})$.
\end{theorem}

Theorem \ref{thm:CZ-weighted-Lp-estimate} would follow if we can show that 
\begin{align} \label{proof:thm:CZ-weighted-Lp-estimate-part1} 
\|T_{1}f\|_{L^p(\omega)} \leq C \|f\|_{L^p(\omega)}, \end{align} 
and that there exists some $q \in (1, p)$ such that 
\begin{align} \label{proof:thm:CZ-weighted-Lp-estimate-part2} 
\| T_{2} f \|_{L^p(\omega)} \leq C \frac{\beta^{n(1-\frac{1}{p})}}{n(1 - \frac{q}{p}) - \beta} \ \|f\|_{L^{1}((M\omega)^{\frac{1}{p}})}.  
\end{align} 

The proof of the weighted $L^p$-estimate \eqref{proof:thm:CZ-weighted-Lp-estimate-part1} of $T_1$ can be done following the classical proof of weighted boundedness of the Calder\'{o}n-Zygmund singular integral and incorporating the support of the kernel of $T_1$. In doing so, we can follow the proof of Theorem 2.1.6 of \cite{Lu-Ding-Yan-book-singular-int-2007} to show that $M (T_1 f) \in L^p(\omega)$. Then, as shown in the proof of Theorem 2.1.6 of \cite{Lu-Ding-Yan-book-singular-int-2007}, inequality \eqref{proof:thm:CZ-weighted-Lp-estimate-part1} will follow from Lemma \ref{lem:weighted-Fefferman-Stein} and the following lemma. 

\begin{lemma} \label{proof:thm:CZ-weighted-Lp-estimate-part1-Lemma1} 
For $1<s<p $ we have $M^{\#}(T_1f)(x) \lesssim_s (M(|f|^s)(x))^{\frac{1}{s}}. $ 
\end{lemma}
\begin{proof}
Suppose $Q$ is a cube containing $x$, and let $B$ be the ball with the same center and of diameter equal to the length of the diagonal of $Q$. Let $f_1 = f \, {\mathbbm{1}}_{2B}$ and $f_2 = f \, {\mathbbm{1}}_{(2B)^c}$, so that $f=f_1+f_2$. Now, it suffices to show that $$ \frac{1}{|Q|} \int_Q |T_1f(y)-T_1f_2(x)| \, dy \leq C (M(|f|^s)(x))^{\frac{1}{s}}, $$
which will follow from the following two estimates: 
\begin{align*} 
\frac{1}{|Q|}\int_Q |T_1f_1(y)| \, dy & \leq C (M(|f|^s)(x))^{\frac{1}{s}}, \\ 
\text{and} \quad \frac{1}{|Q|}\int_Q |T_1f_2(y)-T_1f_2(x)| \, dy & \leq C (M(|f|^s)(x))^{\frac{1}{s}}. 
\end{align*} 

The first inequality is a simple consequence of H\"{o}lder's inequality and the known $L^s$-boundedness of $T_1$ (Lemma \ref{thm2.1:Chen-Guo-Extension-CZ-JFA-2021}). On the other hand, with $d$ denoting the radius of $B$, 
\begin{align*}
& \int_Q |T_1f_2(y)-T_1f_2(x)| \, dy \\
& \leq \int_Q \sum_{k=1}^{\infty}\int_{2^k d <|x-z|\leq 2^{k+1}d}\left| \frac{\Omega(y-z)}{|y-z|^{n-\beta}}\chi_{\beta}(|y-z|)-\frac{\Omega(x-z)}{|x-z|^{n-\beta}}\chi_{\beta}(|x-z|) \right| |f(z)| \, dz \, dy \\ 
& \lesssim \int_Q \left( \sum_{k=1}^{\infty} \left\{ \frac{|x-y|}{(2^k d)^{n+1}} \left\| \Omega \right\|_{L^\infty (S^{n-1})} + \frac{1}{(2^k d)^n} \omega_{\infty} \left(2^{-k+1} \right) \right\} \int_{|x-z|\leq 2^{k+1}d}|f(z)| \, dz \right) dy \\ 
& \lesssim \int_Q \left\{ Mf(x) + Mf(x) \sum_{k=1}^{\infty} \omega_{\infty} \left(2^{-k+1} \right) \right\} dy, \\ 
& \lesssim |Q| \, Mf(x), 
\end{align*}
where the second inequality follows from Lemma \ref{lem:Main-lemma}, and convergence of $\sum_{k=1}^{\infty} \omega_{\infty} \left(2^{-k+1} \right)$ is a consequence of assumption \eqref{conditions:Dini-Omega-function} (similar to the one seen during the estimation of $E_2$ in the proof of part \ref{part(i)-T_1-proof-thm:CZ-Lipschitz-space} in the proof of Theorem \ref{subsec:CZ-Lipschitz}). 

Altogether, we have 
$$ \frac{1}{|Q|} \int_Q |T_1f_2(y)-T_1f_2(x)| \, dy \lesssim Mf(x) \lesssim (M(|f|^s)(x))^{\frac{1}{s}}, $$
which completes the proof of the Lemma \ref{proof:thm:CZ-weighted-Lp-estimate-part1-Lemma1}.
\end{proof}

Next, we shall prove inequality \eqref{proof:thm:CZ-weighted-Lp-estimate-part2}. For this we shall make use of the following important property of Muckenhoupt weights, that for given $\omega \in A_p$, there exists $q \in (1, p)$ such that $\omega \in A_q$, and then there exists a constant $C>0$ (depending on $\omega$ and $q$) such that 
$$\omega(r B) \leq C r^{n q} \ \omega(B)$$
for any cube $Q$ and $r>1$. 

Now, applying Minkowski's inequality and taking $0<\beta < n(1-\frac{q}{p})$, we can estimate $T_2$ as follows: 
\begin{align*}
& \left( \int_{\mathbb{R}^n} |T_2f(x)|^p \omega(x) \, dx \right)^{\frac{1}{p}} \\
& \leq \int_{\mathbb{R}^n} |f(y)| \left( \int_{\mathbb{R}^n} \frac{|\Omega(x-y)|^p}{|x-y|^{(n-\beta)p}} (1-\chi_\beta (|x-y|))^p \omega(x) \, dx  \right)^{\frac{1}{p}} dy \\ 
& \lesssim \int_{\mathbb{R}^n} |f(y)| \left(\sum_{j=0}^{\infty} \int_{\frac{2^j}{\beta} \leq |x-y| \leq  \frac{2^{j+1}}{\beta}} \frac{\omega(x)}{|x-y|^{(n-\beta)p}}  \, dx  \right)^{\frac{1}{p}} dy \\
& \leq \int_{\mathbb{R}^n} |f(y)| \left(\sum_{j=0}^{\infty} \left( \frac{2^j}{\beta} \right)^{-(n-\beta)p} \omega \left(B(y,2^{j+1}\frac{1}{\beta}) \right) \right)^{\frac{1}{p}} dy \\
& \lesssim \left(\frac{1}{\beta} \right)^{\beta} \beta^{n} \int_{\mathbb{R}^n} |f(y)| \left(\sum_{j=0}^{\infty} (2^j)^{-(n-\beta)p} \,  \left(2^{j+1}\right)^{nq}\omega\left(B(y,\frac{1}{\beta})\right) \right)^{\frac{1}{p}} dy \\
& \lesssim \beta^{n}\sum_{j=0}^{\infty} (2^j)^{-(n-\beta)+n\frac{q}{p}} \int_{\mathbb{R}^n}|f(y)|\left(\omega(B(y,\frac{1}{\beta}))\right)^{\frac{1}{p}} dy \\
& = \frac{\beta^{n(1-\frac{1}{p})}}{1- 2^{-(n-\beta)+n\frac{q}{p}}} \int_{\mathbb{R}^n}|f(y)|\left(\beta^n\int_{B(y,\frac{1}{\beta})}\omega(x) \, dx\right)^{\frac{1}{p}} dy \\
& \lesssim \frac{\beta^{n(1-\frac{1}{p})}}{n(1 - \frac{q}{p}) - \beta}  \int_{\mathbb{R}^n} |f(y)| M\omega(y)^{\frac{1}{p}} \, dy,  
\end{align*}
completing the proof of inequality \eqref{proof:thm:CZ-weighted-Lp-estimate-part2} and with this the proof of Theorem \ref{thm:CZ-weighted-Lp-estimate} is completed. 

%%%%%%%%%%%%%%%%%%%%%%%%%%%%%%%%%%%%%%
%%%%%%%%%%%%%%%%%%%%%%%%%%%%%%%%%%%%%%
%%%%%%%%%%%%%%%%%%%%%%%%%%%%%%%%%%%%%%
%%%%%%%%%%%%%%%%%%%%%%%%%%%%%%%%%%%%%%
%%%%%%%%%%%%%%%%%%%%%%%%%%%%%%%%%%%%%%
%%%%%%%%%%%%%%%%%%%%%%%%%%%%%%%%%%%%%%
%%%%%%%%%%%%%%%%%%%%%%%%%%%%%%%%%%%%%%
%%%%%%%%%%%%%%%%%%%%%%%%%%%%%%%%%%%%%%
%%%%%%%%%%%%%%%%%%%%%%%%%%%%%%%%%%%%%%
%%%%%%%%%%%%%%%%%%%%%%%%%%%%%%%%%%%%%%
%%%%%%%%%%%%%%%%%%%%%%%%%%%%%%%%%%%%%%
%%%%%%%%%%%%%%%%%%%%%%%%%%%%%%%%%%%%%%
%%%%%%%%%%%%%%%%%%%%%%%%%%%%%%%%%%%%%%

\section{Boundedness of commutators of CZ type operators} \label{sec:proofs-commutator} 
This section is devoted to proving Theorems \ref{thm:commutator-b-BMO-Lp-estimate}, \ref{thm:commutator-b-Lipschitz-Lp-Lq-estimate} and \ref{thm:commutator-b-Lipschitz-Hp-Lq-estimate} concerning estimates for commutator operator $[b,T_{\beta}]$ defined in \eqref{def:extension-commutator-operator}. 

With the cut-off function $\chi$ as in \eqref{def:cut-off-function}, we decompose $[b,T_{\beta}] f$ as follows: 
$$ [b,T_{\beta}]f(x) = [b,T_{\beta}]_1f(x) + [b,T_{\beta}]_2f(x),$$ 
where 
\begin{align*} 
[b,T_{\beta}]_1f(x) = \int_{\mathbb{R}^n} \frac{\Omega(x-y)}{|x-y|^{n-\beta}} \chi_{\beta}(|x-y|)[b(x)-b(y)] \, f(y) \, dy, 
\end{align*} 
for $0 < \beta < n$. 

%%%%%%%%%%%%%%%%%%%%%%%%%%%%%%%%%%%
%%%%%%%%%%%%%%%%%%%%%%%%%%%%%%%%%%%

\subsection{\texorpdfstring{$(L^p, L^p)$}{}-boundedness: Proof of Theorem \ref{thm:commutator-b-BMO-Lp-estimate}} \label{subsec:commutator-BMO} 

In view of \eqref{inequality:Sharp-maximal}, we will be done if we could show that 
\begin{align} \label{est:sharp-max-commutator-part1}
\| M^{\#} \left( [b,T_{\beta}]_1f \right) \|_{L^p} \leq C \|b\|_{BMO} \|f\|_{L^p}, 
\end{align}
and 
\begin{align} \label{est:sharp-max-commutator-part2}
\| M^{\#} \left( [b,T_{\beta}]_2 f \right) \|_{L^p} & \leq C \|b\|_{BMO}\left\{ \frac{\beta^{n(1-\frac{1}{p})}}{\sqrt[p]{n(p-1)-\beta p}} \|f\|_{L^1} + \frac{\beta^{l}}{(l-\beta)^{1 - \frac{l}{n}}} \|f\|_{L^r} \right\}, 
\end{align}
with indices same as in the statement of Theorem \ref{thm:commutator-b-BMO-Lp-estimate}. 

\medskip \noindent \textbf{\underline{Proof of estimate \eqref{est:sharp-max-commutator-part1}}:} 
For estimating $M^{\#} \left( [b,T_{\beta}]_1f \right)$, we suitably adapt the proof of Theorem 2.4.1 from \cite{Lu-Ding-Yan-book-singular-int-2007}. Fix a cube $Q$ with center $x_0$ and write 
\begin{align*}
[b,T_{\beta}]_1f(x) & = (b(x)-b_Q)T_1f(x) - T_1((b-b_Q)f {\mathbbm{1}}_{2Q})(x) - T_1((b-b_Q)f {\mathbbm{1}}_{(2Q)^c})(x) \\ 
& =: a_1(x) - a_2(x) - a_3(x). 
\end{align*} 
Then, in view of \eqref{ineq:sharp-maximal-function-pointwise}, we have
\begin{align*}
M^{\#} \left( [b,T_{\beta}]_1f \right)(x_0) & \lesssim \sup_Q \frac{1}{|Q|} \int_Q \left| [b,T_{\beta}]_1f(x) - a_3 (x_0)\right| dx \\
& \lesssim \sup_Q \frac{1}{|Q|} \left[\int_{Q}|a_1(x)| \, dx + \int_{Q}|a_2(x)| \, dx +  \int_{Q}|a_3(x)-a_3(x_0)| \, dx \right].
\end{align*}

For $a_1 (x)$, since $1<r<p$, applying H\"older's inequality, we get 
\begin{align} \label{est:sharp-max-commutator-part1-a1}
\frac{1}{|Q|} \int_{Q}|a_1(x)| \, dx & \leq \left( \frac{1}{|Q|} \int_Q \left| b(x)-b_Q \right|^{r'} dx \right)^{1/r'} \left( \frac{1}{|Q|} \int_Q \left| T_1f(x) \right|^r dx \right)^{1/r} \\ 
\nonumber & \lesssim \|b\|_{BMO} \ \left(M(|T_1 f|^r)(x_0) \right)^{\frac{1}{r}}. 
\end{align} 

For $a_2 (x)$, choose $s, t > 1$ such that $st=r$, and using Theorem \ref{thm2.1:Chen-Guo-Extension-CZ-JFA-2021}, we get 
\begin{align} \label{est:sharp-max-commutator-part1-a2}
\frac{1}{|Q|}\int_{Q}|a_2(x)| \, dx & \leq \left( \frac{1}{|Q|}\int_{Q}|T_1((b-b_Q)f {\mathbbm{1}}_{2Q})(x)|^s  \, dx \right)^\frac{1}{s} \\
\nonumber & \lesssim \left(\frac{1}{|Q|}\int_{2Q}|b(x)-b_Q|^s |f(x)|^s  \, dx \right)^{\frac{1}{s}}  \\
\nonumber & \leq \left(\frac{1}{|Q|} \int_{2Q}|b(x)-b_Q|^{s t^{\prime}}  \, dx \right)^{\frac{1}{s t^{\prime}}} \left(\frac{1}{|Q|}\int_{2Q}|f(x)|^{s t}  \, dx \right)^{\frac{1}{s t}}\\
\nonumber & \lesssim \|b\|_{BMO} \left(M(|f|^r)(x_0) \right)^{\frac{1}{r}}. 
\end{align} 

Finally, for $a_3 (x), $ using Remark \ref{rem:lem:Main-lemma}, we have for any $x \in Q$, 
\begin{align} \label{est:sharp-max-commutator-part1-a3}
& |a_3(x)-a_3(x_0)| \\ 
\nonumber & \leq \sum_{k=1}^{\infty}\int_{2^{k+1}Q\setminus 2^k Q}\left| \frac{\Omega(x-z)}{|x-z|^{n-\beta}}\chi_{\beta}(|x-z|)-\frac{\Omega(x_0-z)}{|x_0-z|^{n-\beta}}\chi_{\beta}(|x_0-z|) \right||b(z)-b_Q| |f(z)| \, dz \\ 
\nonumber & \lesssim \left\| \Omega \right\|_{L^\infty (S^{n-1})} \sum_{k=1}^{\infty}\int_{2^{k+1}Q\setminus 2^k Q}\frac{|x-x_0|}{|x_0-z|^{n+1}}|b(z)-b_Q||f(z)| \, dz \\
\nonumber & \quad +  \sum_{k=1}^{\infty} \omega_{\infty} \left(\frac{\kappa_n}{2^{k-1}} \right) \int_{2^{k+1}Q\setminus 2^k Q}\frac{1}{|x_0-z|^{n}}|b(z)-b_Q||f(z)| \, dz \\
\nonumber & \lesssim \left(\sum_{k=1}^{\infty} \frac{1}{2^k}\frac{1}{|2^{k+1} Q|}\int_{2^{k+1}Q}|b(z)-b_Q|^{r^{\prime}} \, dz \right)^{\frac{1}{r^{\prime}}}\left(\sum_{k=1}^{\infty}\frac{1}{2^k}\frac{1}{|2^{k+1} Q|}\int_{2^{k+1}Q}|f(z)|^{r} \, dz \right)^{\frac{1}{r}} \\
\nonumber & \quad \quad + \left(\sum_{k=1}^{\infty} \omega_{\infty}\left(\frac{\kappa_n}{2^{k-1}}\right) \frac{1}{|2^{k+1} Q|}\int_{2^{k+1}Q}|b(z)-b_Q|^{r^{\prime}} \, dz \right)^{\frac{1}{r^{\prime}}} \\
\nonumber & \quad \quad \quad \times \left(\sum_{k=1}^{\infty} \omega_{\infty} \left(\frac{\kappa_n}{2^{k-1}}\right) \frac{1}{|2^{k+1} Q|} \int_{2^{k+1}Q}|f(z)|^{r} \, dz \right)^{\frac{1}{r}} \\
\nonumber & \lesssim \|b\|_{BMO} \ \left(M(| f|^r)(x_0) \right)^{\frac{1}{r}}.
\end{align} 

Combining estimates \eqref{est:sharp-max-commutator-part1-a1}, \eqref{est:sharp-max-commutator-part1-a2}, \eqref{est:sharp-max-commutator-part1-a3} together, and taking the supremum over all cubes $Q$ with center $x_0$, if we take the $L^p$-norm in $x_0$-variable, then using Theorem \ref{thm2.1:Chen-Guo-Extension-CZ-JFA-2021}, we conclude that estimate \eqref{est:sharp-max-commutator-part1} holds true.

\medskip \noindent \textbf{\underline{Proof of estimate \eqref{est:sharp-max-commutator-part2}}:} As earlier, fix a cube $Q$ with center $x_0$ and write 
\begin{align*}
[b,T_{\beta}]_2f(x) & = (b(x)-b_Q)T_2f(x)-T_2((b-b_Q)f{\mathbbm{1}}_{2Q})(x)-T_2((b-b_Q)f{\mathbbm{1}}_{(2Q)^c})(x) \\
& =: A_1(x) - A_2(x) - A_3(x). 
\end{align*}
Same as in the case of $a_1 (x)$, for $A_1 (x)$ too, applying H\"older's inequality, we get 
\begin{align*} 
\frac{1}{|Q|}\int_{Q}|A_1(x)| \, dx  \leq C \|b\|_{BMO} (M(|T_2f|^r)(x_0))^{\frac{1}{r}}. 
\end{align*}
Therefore, for $0<\beta<l<n(1-\frac{1}{p})$, we can make use of Lemma \ref{lem2.2:Yu-Jiu-Li-CZ-JFA-2021} together with taking supremum over all cubes $Q$ with center $x_0$, we get 
\begin{align} \label{est:sharp-max-commutator-part2-A1} 
\left\| \sup_Q \frac{1}{|Q|} \int_{Q}|A_1(x)| \, dx \right\|_{L^p(\, dx_0)} & \lesssim \frac{\beta^{n(1-\frac{1}{p})}}{\sqrt[p]{n(p-1)-\beta p}} \|b\|_{BMO} \|f\|_{L^1}.
\end{align}

For $A_2 (x)$, applying H\"older's inequality with exponent $\frac{n}{n-l}$ and its H\"older conjugate, and making use of Lemma \ref{lem2.2:Yu-Jiu-Li-CZ-JFA-2021}, as $\beta<l$ and choosing $1<u<r$, we get 
\begin{align*}
\frac{1}{|Q|}\int_{Q} |A_2(x)|  \, dx & \leq  \left( \frac{1}{|Q|} \int_{Q}|T_2((b-b_Q)f{\mathbbm{1}}_{2Q})(x)|^\frac{n}{n-l}  \, dx\right)^{\frac{n-l}{n}} \\
& \lesssim \frac{\beta^{l}}{\left( \frac{n l}{n-l}-\frac{n\beta}{n-l}\right)^{1-\frac{l}{n}}}\frac{1}{|Q|^{1-\frac{l}{n}}} \int_{2 Q} |((b-b_Q)f{\mathbbm{1}}_{2Q})(x)|  \, dx \\
& \lesssim \frac{\beta^{l}}{(l-\beta)^{1-\frac{l}{n}}}\frac{1}{|Q|^{1-\frac{l}{n}}} \left(\int_{2Q}|b(y)-b_Q|^{u^{\prime}} \, dy \right)^{\frac{1}{u^{\prime}}}\left(\int_{2Q}|f(y)|^u  \, dy\right)^{\frac{1}{u}} \\
& \lesssim \frac{\beta^{l}}{(l-\beta)^{1-\frac{l}{n}}} \|b\|_{BMO} \left(\frac{1}{|Q|^{1-\frac{l u}{n}}} \int_{2Q}|f(y)|^u  \, dy\right)^{\frac{1}{u}} \\
& \lesssim \frac{\beta^{l}}{(l-\beta)^{1-\frac{l}{n}}} \|b\|_{BMO} \ [f]_{l , u}^{*}(x_0), 
\end{align*} 
and therefore, using Lemma \ref{lem:fractional-maximal-Lp-Lq-bound}, we get 
\begin{align} \label{est:sharp-max-commutator-part2-A2} 
\left\| \sup_Q \frac{1}{|Q|} \int_{Q}|A_2(x)| \, dx \right\|_{L^p(\, dx_0)} & \lesssim \frac{\beta^{l}}{(l-\beta)^{1-\frac{l}{n}}} \|b\|_{BMO} \|f\|_{L^r}. 
\end{align}

Finally, we estimate the part for $A_3 (x)$ as follows. We make use of the mean-value theorem for $\left| \frac{1}{|x-y|^{n-\beta}} - \frac{1}{|x_0-y|^{n-\beta}} \right|$ and $\left| (1-\chi_{\beta}(|x-y|)) - (1-\chi_{\beta}(|x_0-y|)) \right|$ (as also done in Section \ref{sec:proofs-FCZ-Lipschitz-spaces}), and the $L^{\infty}$-boundedness of $\Omega$, to decompose $|A_3(x)-A_3(x_0)|$ in the following manner. For any $x \in Q$,  
\begin{align*}
& |A_3(x)-A_3(x_0)| \\
& \leq \int_{\mathbb{R}^n \setminus 2Q} \left|\frac{\Omega(x-y)}{|x-y|^{n-\beta}}(1-\chi_{\beta}(|x-y|))-\frac{\Omega(x_0-y)}{|x_0-y|^{n-\beta}}(1-\chi_{\beta}(|x_0-y|)) \right| \, |b(y)-b_Q||f(y)| \, dy \\ 
& \lesssim \int_{\mathbb{R}^n \setminus 2Q} \frac{|x-x_0|}{|x-y|^{n-\beta+1}}(1-\chi_{\beta}(|x-y|)) \,  |b(y)-b_Q| \, |f(y)| \, dy \\ 
& \quad + \int_{\mathbb{R}^n \setminus 2Q} \frac{|x-x_0|}{|x-y|^{n-\beta+1}}{\mathbbm{1}}_{\{|x-y| \sim \frac{1}{\beta}\}} \, |b(y)-b_Q| \, |f(y)| \, dy \\
& \quad + \int_{\mathbb{R}^n \setminus 2Q} \frac{(1-\chi_{\beta}(|x-y|))}{|x_0-y|^{n-\beta}}|\Omega(x-y)-\Omega(x_0-y)| |b(y)-b_Q| \, |f(y)| \, dy.
\end{align*}

For the first part of the above expression, we have
\begin{align*}
& \int_{\mathbb{R}^n \setminus 2Q} \frac{|x-x_0|}{|x-y|^{n-\beta+1}}(1-\chi_{\beta}(|x-y|))|b(y)-b_Q||f(y)| \, dy \\
& \leq \left(\int_{(2Q)^c}\frac{|x-x_0|}{|x-y|^{n+1}}|b(y)-b_Q|^{r^{\prime}} \, dy \right)^{\frac{1}{r^{\prime}}} \left(\int_{(2Q)^c}\frac{|x-x_0|(1-\chi_{\beta}(|x-y|))}{|x-y|^{n+1-\beta r}}|f(y)|^r  \, dy \right)^{\frac{1}{r}} \\
& \lesssim \left(\sum_{j=1}^{\infty}\int_{2^{j+1}Q\setminus 2^j Q}\frac{|x-x_0|}{|x-y|^{n+1}}|b(y)-b_Q|^{r^{\prime}}  \, dy \right)^{\frac{1}{r^{\prime}}} \left(\int_{|x-y| \geq \frac{1}{\beta}} \frac{1}{|x-y|^{n-\beta r}}|f(y)|^r \, dy \right)^{\frac{1}{r}} \\
& \leq \left(\sum_{j=1}^{\infty}\frac{1}{2^j}\frac{1}{|2^{j+1}Q|}\int_{2^{j+1}Q}|b(y)-b_Q|^{s^{\prime}}  \, dy \right)^{\frac{1}{r^{\prime}}}\left(\int_{|x_0-y|\gtrsim \frac{1}{\beta}}\frac{1}{|x_0-y|^{n-\beta r}}|f(y)|^r  \, dy \right)^{\frac{1}{r}} \\
& \lesssim \|b\|_{BMO}\left(\int_{\mathbb{R}^n}\frac{{\mathbbm{1}}_{\{|x_0-y| \gtrsim \frac{1}{\beta}\}}}{|x_0-y|^{n-\beta r}}|f(y)|^r  \, dy \right)^{\frac{1}{r}}.
\end{align*}

Similarly, for the second part, we have 
$$ \int_{\mathbb{R}^n \setminus 2Q} \frac{|x-x_0|}{|x-y|^{n-\beta+1}}{\mathbbm{1}}_{\{|x-y|\sim \frac{1}{\beta}\}} |b(y)-b_Q||f(y)| \, dy \lesssim \|b\|_{BMO} \left(\int_{\mathbb{R}^n}\frac{{\mathbbm{1}}_{\{|x_0-y|\sim \frac{1}{\beta}\}} }{|x_0-y|^{n-\beta r}}|f(y)|^r  \, dy \right)^{\frac{1}{r}}. $$ 

Finally, for the third part, we have 
\begin{align*}
& \int_{\mathbb{R}^n \setminus 2Q} \frac{(1-\chi_{\beta}(|x-y|))}{|x_0-y|^{n-\beta}}|\Omega(x-y)-\Omega(x_0-y)| |b(y)-b_Q||f(y)| \, dy \\
& \lesssim \sum_{j=1}^{\infty} \left( \int_{2^{j+1}Q \setminus 2^jQ}\frac{|\Omega(x-y)-\Omega(x_0-y)|^{r^{\prime}}}{|x-y|^{n}}|b(y)-b_Q|^{r^{\prime}} \, dy \right)^{\frac{1}{r^{\prime}}} \left(\int_{|x-y| \geq \frac{1}{\beta}} \frac{|f(y)|^r}{|x-y|^{n-\beta r}} \, dy \right)^{\frac{1}{r}} \\
& \lesssim \sum_{j=1}^{\infty} \omega_{\infty} \left(\frac{\kappa_n}{2^{j-1}} \right) \left(\frac{1}{|2^{j+1}Q|}\int_{2^{j+1}Q}|b(y)-b_Q|^{r^{\prime}}  \, dy \right)^{\frac{1}{r^{\prime}}} \left(\int_{|x_0-y| \gtrsim \frac{1}{\beta}} \frac{1}{|x_0-y|^{n-\beta r}}|f(y)|^r  \, dy \right)^{\frac{1}{r}} \\
& \lesssim \|b\|_{BMO} \left( \int_{\mathbb{R}^n} \frac{{\mathbbm{1}}_{\{|x_0-y| \gtrsim \frac{1}{\beta}\}}}{|x_0-y|^{n-\beta r}}|f(y)|^r  \, dy \right)^{\frac{1}{r}}. 
\end{align*} 

Taking all three estimates together, and making use of  $0 < \beta < l = n(\frac{1}{r}-\frac{1}{p})$, we get 
\begin{align} \label{est:sharp-max-commutator-part2-A3} 
\left\| \sup_Q \frac{1}{|Q|}\int_{Q}|A_3(x)-A_3(x_0)|  \, dx \right\|_{L^p(d x_0)} & \lesssim \|b\|_{BMO} \left\| \left(\int_{\mathbb{R}^n}\frac{{\mathbbm{1}}_{\{|x_0-y| \geq \frac{1}{\beta}\}}}{|x_0-y|^{n-\beta r}} |f(y)|^r  \, dy \right)^{\frac{1}{r}}\right\|_{L^p(d x_0)} \\ 
\nonumber & \lesssim \frac{\beta^{n(\frac{1}{r}-\frac{1}{p})}}{\sqrt[p]{n(\frac{p}{r}-1)-\beta p}} \|b\|_{BMO} \|f\|_{L^r} \\
\nonumber & \lesssim \frac{\beta^{l}}{\sqrt[p]{l-\beta}} \|b\|_{BMO} \|f\|_{L^r}.
\end{align} 

Combining estimates \eqref{est:sharp-max-commutator-part2-A1}, \eqref{est:sharp-max-commutator-part2-A2}, and \eqref{est:sharp-max-commutator-part2-A3} together, we conclude that estimate \eqref{est:sharp-max-commutator-part2} holds true, and this completes the proof of Theorem \ref{thm:commutator-b-BMO-Lp-estimate}.

%%%%%%%%%%%%%%%%%%%%%%%%%%%%%%%%%%%
%%%%%%%%%%%%%%%%%%%%%%%%%%%%%%%%%%%

\subsection{\texorpdfstring{$(L^p, L^q)$}{}-boundedness: Proof of Theorem \ref{thm:commutator-b-Lipschitz-Lp-Lq-estimate}}  \label{subsec:commutator-Lp-Lq-spaces} 

As $b \in Lip_{\gamma}(\mathbb{R}^n)$, the boundedness of $[b,T_{\beta}]_2f$ is easy. In fact, with $0 < \beta < n(1-\frac{1}{q}) - \gamma$, we have
\begin{align} \label{est1:proof-thm:commutator-b-Lipschitz-Lp-Lq-estimate}
\|[b,T_{\beta}]_2f\|_{L^q} & = \left(\int_{\mathbb{R}^n}\left|\int_{\mathbb{R}^n}\frac{\Omega(x-y)}{|x-y|^{n-\beta}}(1-\chi_{\beta}(|x-y|))[b(x)-b(y)]f(y) \, dy \right|^q dx \right)^{\frac{1}{q}} \\
\nonumber & \leq  \int_{\mathbb{R}^n}|f(y)|\left(\int_{|x-y| \geq \frac{1}{\beta}} \frac{|b(x)-b(y)|^q}{|x-y|^{(n-\beta)q}} \, dx \right)^{\frac{1}{q}} dy \\
\nonumber & \lesssim \|b\|_{Lip_{\gamma}} \int_{\mathbb{R}^n} |f(y)| \left(\int_{|x-y| \geq \frac{1}{\beta}} \frac{1}{|x-y|^{(n-\beta-\gamma)q}} \, dx \right)^{\frac{1}{q}} dy \\
\nonumber & \lesssim \frac{\beta^{n(1-\frac{1}{q})-\gamma}}{\sqrt[q]{n(q-1)-(\beta-\gamma)q}} \|b\|_{Lip_{\gamma}} \, \|f\|_{L^1}.
\end{align}

%%%%%%%%%%%%%%%%%%%%%
%%%%%%%%%%%%%%%%%%%%%
%%%%%%%%%%%%%%%%%%%%%
%%%%%%%%%%%%%%%%%%%%%

To analyse $[b,T_{\beta}]_1 f$, we again take inputs from the proof of Lemma 2.4.1 of \cite{Lu-Ding-Yan-book-singular-int-2007}. In view of inequality \eqref{inequality:Sharp-maximal}, Lemma \ref{lem:fractional-maximal-Lp-Lq-bound}, and Theorem \ref{thm2.1:Chen-Guo-Extension-CZ-JFA-2021}, the proof of Theorem \ref{thm:commutator-b-Lipschitz-Lp-Lq-estimate} will be completed once we show that 
\begin{align} \label{sharp maximal extimate of T_1}
    M^{\#}([b,T_{\beta}]_1f)(x_0) \lesssim \|b\|_{Lip_{\gamma}}\left([T_1 f]_{\gamma , l}^{*} (x_0) + [f]_{\gamma , l}^{*} (x_0) \right),
\end{align}
which we now proceed to show. 

We first decompose $[b,T_{\beta}]_1 f$ into three parts, as done in the proof of estimate \eqref{est:sharp-max-commutator-part1}. Namely, fixing a cube $Q$ with center $x_0$, we write 
\begin{align*}
[b,T_{\beta}]_1f(x) & = (b(x)-b_Q)T_1f(x)-T_1((b-b_Q)f {\mathbbm{1}}_{2Q})(x)-T_1((b-b_Q)f{\mathbbm{1}}_{(2Q)^c})(x) \\
& =: a_1(x)-a_2(x)-a_3(x),
\end{align*}
and fix some $1 < l < p$. 

For $a_1(x)$, applying H\"older's inequality, and as $b \in Lip_{\gamma}$, we get 
\begin{align} \label{Lipschitz a_1 estimate}
\frac{1}{|Q|}\int_{Q}|a_1(x)| \, dx & \leq \left(\frac{1}{|Q|}\int_Q |b(x)-b_Q|^{l^\prime} \, dx \right)^{\frac{1}{l^{\prime}}}\left(\frac{1}{|Q|}\int_{Q}|T_1f(x)|^l \, dx \right)^{\frac{1}{l}} \\
\nonumber & \lesssim \|b\|_{Lip_{\gamma}} |Q|^{\frac{\gamma}{n}}\left(\frac{1}{|Q|} \int_{Q} |T_1f(x)|^l \, dx \right)^{\frac{1}{l}} \\
\nonumber & \leq \|b\|_{Lip_{\gamma}} \, [T_1 f]_{\gamma , l}^{*}(x_0).
\end{align}

For $a_2(x)$, applying again H\"older's inequality and Theorem \ref{thm2.1:Chen-Guo-Extension-CZ-JFA-2021}, we get
\begin{align} \label{Lipschitz a_2 estimate}
\frac{1}{|Q|}\int_{Q}|a_2(x)| \, dx & \leq \left( \frac{1}{|Q|}\int_{Q}|T_1((b-b_Q)f{\mathbbm{1}}_{2Q})(x)|^l  \, dx \right)^l \\
\nonumber & \lesssim \left(\frac{1}{|Q|}\int_{2Q}|b(x)-b_Q|^l |f(x)|^l  \, dx \right)^{\frac{1}{l}} \\
\nonumber & \leq C \|b\|_{Lip_{\gamma}} \, [f]_{\gamma , l}^{*}(x_0). 
\end{align} 

Finally for $a_3(x)$, using Remark \ref{rem:lem:Main-lemma}, H\"older's inequality, Dini condition \eqref{conditions:Dini-Omega-function}, and the fact that $b \in Lip_{\gamma}(\mathbb{R}^n)$, we get 
\begin{align} \label{Lipschitz a_3 estimate}
& |a_3(x)-a_3(x_0)| \\ 
\nonumber & \lesssim \sum_{k=1}^{\infty}\int_{2^{k+1}Q\setminus 2^k Q}\left| \frac{\Omega(x-z)}{|x-z|^{n-\beta}}\chi_{\beta}(|x-z|)-\frac{\Omega(x_0-z)}{|x_0-z|^{n-\beta}}\chi_{\beta}(|x_0-z|) \right||b(z)-b_Q| |f(z)| \, dz \\
\nonumber & \lesssim \left\| \Omega \right\|_{L^\infty (S^{n-1})} \sum_{k=1}^{\infty}\int_{2^{k+1}Q\setminus 2^k Q}\frac{|x-x_0|}{|x_0-z|^{n+1}}|b(z)-b_Q||f(z)| \, dz  \\
\nonumber & \quad + \sum_{k=1}^{\infty}\omega_{\infty}\left(\frac{\kappa_n}{2^{k-1}} \right)\int_{2^{k+1}Q\setminus 2^k Q}\frac{1}{|x_0-z|^{n}}|b(z)-b_Q||f(z)| \, dz \\
\nonumber & \lesssim \|b\|_{Lip_{\gamma}} \sum_{k=1}^{\infty}\frac{|x-x_0|}{|2^k Q|^{1-\frac{\gamma}{n}+\frac{1}{n}}}\int_{2^{k+1}Q\setminus 2^k Q}|f(z)|  \, dz  \\
\nonumber & \quad + \|b\|_{Lip_{\gamma}} \sum_{k=1}^{\infty}\omega_{\infty}\left(\frac{\kappa_n}{2^{k-1}} \right)\frac{1}{|2^k Q|^{1-\frac{\gamma}{n}}}\int_{2^{k+1}Q\setminus 2^k Q}|f(z)|  \, dz \\
\nonumber & \lesssim \|b\|_{Lip_{\gamma}} \sum_{k=1}^{\infty}\frac{1}{2^k}\frac{1}{|2^k Q|^{1-\frac{\gamma}{n}}}|2^k Q|^{\frac{1}{l^{\prime}}}\left(\int_{2^{k+1}Q}|f(z)|^l  \, dz\right)^{\frac{1}{l}} \\
\nonumber & \quad + \|b\|_{Lip_{\gamma}} \sum_{k=1}^{\infty}\omega_{\infty}\left(\frac{\kappa_n}{2^{k-1}} \right)\frac{1}{|2^k Q|^{1-\frac{\gamma}{n}}}|2^k Q|^{\frac{1}{l^{\prime}}}\left(\int_{2^{k+1}Q}|f(z)|^l  \, dz\right)^{\frac{1}{l}} \\
\nonumber & \lesssim \|b\|_{Lip_{\gamma}} \sum_{k=1}^{\infty}\left(\frac{1}{2^k}+\omega_{\infty}\left(\frac{\kappa_n}{2^{k-1}} \right)\right)\left(\frac{1}{|2^{k+1}Q|^{1-\frac{\gamma l}{n}}}\int_{2^{k+1}Q}|f(z)|^l  \, dz\right)^{\frac{1}{l}} \\
\nonumber & \lesssim \|b\|_{Lip_{\gamma}} [f]_{\gamma , l}^{*}(x_0).
\end{align}

In view of \eqref{ineq:sharp-maximal-function-pointwise}, combining \eqref{Lipschitz a_1 estimate}, \eqref{Lipschitz a_2 estimate}, \eqref{Lipschitz a_3 estimate} and taking supremum over all cubes $Q$ with center $x_0$, we conclude that \eqref{sharp maximal extimate of T_1} holds true.

This completes the proof of Theorem \ref{thm:commutator-b-Lipschitz-Lp-Lq-estimate}.

%%%%%%%%%%%%%%%%%%%%%%%%%%%%%%%%%%%
%%%%%%%%%%%%%%%%%%%%%%%%%%%%%%%%%%%

\subsection{\texorpdfstring{$(H^p, L^q)$}{}-boundedness: Proof of Theorem \ref{thm:commutator-b-Lipschitz-Hp-Lq-estimate}}  \label{subsec:commutator-Hp-Lq-spaces} 

As $0< \beta < n(1-\frac{1}{q}) - \alpha$, we have already seen in \eqref{est1:proof-thm:commutator-b-Lipschitz-Lp-Lq-estimate} that 
$$ \|[b,T_{\beta}]_2f\|_{L^q} \lesssim \frac{\beta^{n(1-\frac{1}{q})-\alpha}}{\sqrt[q]{n(q-1)-(\beta-\alpha)q}} \|b\|_{Lip_{\alpha}} \, \|f\|_{L^1}. $$ 
So, we are only left with estimating $[b,T_{\beta}]_1 f$. For that, we follow ideas from \cite{Lu-Wu-Yang-commutators-Hardy-2002-Science-China}. By atomic decomposition, we only need to prove that there exists a constant $C > 0$ such that 
$$ \| [b,T_{\beta}]_1 \, a \|_{L^q} \leq C \|b\|_{Lip_{\alpha}},$$ 
for any $(p,2,0)$ atom $a$. 

Let $supp(a) \subseteq B = B(x_0 , r)$. Now, 
\begin{align*} 
\|[b,T_{\beta}]_1a\|_{L^q} \leq \left(\int_{|x-x_0|\leq 2r} |[b,T_{\beta}]_1a(x)|^q  \, dx  \right)^{\frac{1}{q}} + \left(\int_{|x-x_0|> 2r} |[b,T_{\beta}]_1a(x)|^q  \, dx  \right)^{\frac{1}{q}} =: I_1 + I_2. 
\end{align*} 

Choose and fix some $\tilde{p}$ and $\tilde{q}$ satisfying $1 < \tilde{p} < \min\{2,\frac{n}{\alpha}\}$ and $\frac{1}{\tilde{q}} = \frac{1}{\tilde{p}}-\frac{\alpha}{n}$. Since $p \leq 1 < \tilde{p}$ and $\frac{1}{p} - \frac{1}{q} = \frac{1}{\tilde{p}} - \frac{1}{\tilde{q}} = \frac{\alpha}{n}$, it follows that $q < \tilde{q}$. Now, using the $(L^{\tilde{p}},  L^{\tilde{q}})$-boundedness of $[b,T_{\beta}]_1$ (proved in subsection 4.2.), H\"older's inequality, and the property of atom $a$, we get 
\begin{align*} 
I_1 = \left(\int_{|x-x_0|\leq 2r} |[b,T_{\beta}]_1a(x)|^q  \, dx  \right)^{\frac{1}{q}} 
& \leq \left(\int_{|x-x_0|\leq 2r} |[b,T_{\beta}]_1a(x)|^{\tilde{q}}  \, dx  \right)^{\frac{1}{\tilde{q}}} \left( \int_{|x-x_0|\leq 2r} dx \right)^{\frac{1}{q}- \frac{1}{\tilde{q}}} \\ 
& \lesssim \|b\|_{Lip_{\alpha}} \|a\|_{L^{\tilde{p}}} \, r^{n(\frac{1}{q}-\frac{1}{\tilde{q}})} \\ 
& \lesssim \|b\|_{Lip_{\alpha}} \|a\|_{L^2} \, r^{n(\frac{1}{p}-\frac{1}{2})} \\
& \lesssim \|b\|_{Lip_{\alpha}}.
\end{align*}

In order to estimate $I_2$, note that 
$$ |[b,T_{\beta}]_1a(x)| \leq |(b(x)-b(x_0))T_1a(x)| + |T_1((b-b(x_0))a)(x)|.$$

For the first term on the right hand side of the above inequality, using the vanishing condition of atom $a$, we get 
\begin{align*}
& (b(x)-b(x_0))T_1a(x) \\
& = (b(x)-b(x_0)) \int_{B}\left\{ \frac{\Omega(x-y)}{|x-y|^{n-\beta}}\chi_\beta (|x-y|)-\frac{\Omega(x-x_0)}{|x-x_0|^{n-\beta}}
\chi_\beta (|x-x_0|)\right\} a(y)  \, dy, 
\end{align*}
and therefore, using Lemma \ref{lem:Main-lemma} and the definition of atom, we get that for $|x - x_0| > 2r$, 
\begin{align*}
& |(b(x)-b(x_0))T_1a(x)| \\
& \lesssim \|b\|_{Lip_{\alpha}} |x-x_0|^{\alpha}\left\{\int_B\left\{ \frac{|x_0-y|}{|x-x_0|^{n+1}} \left\| \Omega \right\|_{L^\infty (S^{n-1})} +\frac{1}{|x-x_0|^n} \omega_{\infty} \left( 2 \frac{|x_0-y|}{|x-x_0|} \right)\right\} |a(y)|  \, dy \right\} \\
& \lesssim \|b\|_{Lip_{\alpha}}\left\{ |x-x_0|^{\alpha-n-1} \, r^{1+n(1-\frac{1}{p})} + |x-x_0|^{\alpha-n}\omega_{\infty} \left( 2 \frac{r}{|x-x_0|} \right) r^{n(1-\frac{1}{p})}\right\} \\ 
& \lesssim \|b\|_{Lip_{\alpha}} \left\{ |x-x_0|^{-n} \, r^{\alpha+n(1-\frac{1}{p})} + |x-x_0|^{\alpha-n} \omega_{\infty} \left( 2 \frac{r}{|x-x_0|} \right) r^{n(1-\frac{1}{p})}\right\} .
\end{align*}

On the other hand, for $|x - x_0| > 2r$, 
\begin{align*}
|T_1((b-b(x_0)a)(x)| & = \left| \int_B \frac{\Omega(x-y)}{|x-y|^{n}} \left\{ |x-y|^{\beta} \chi_\beta (|x-y|) \right\} (b(y)-b(x_0))a(y)  \, dy \right| \\ 
& \lesssim \left( \frac{2}{\beta} \right)^{\beta} \|b\|_{Lip_{\alpha}} |x-x_0|^{-n} \int_B |y-x_0|^{\alpha}|a(y)| \, dy \\
& \lesssim \|b\|_{Lip_{\alpha}} |x-x_0|^{-n} \, r^{\alpha+n(1-\frac{1}{p})}. 
\end{align*} 

Combining the above estimates, we get 
\begin{align*} 
I_2 & = \left(\int_{|x-x_0|> 2r} |[b,T_{\beta}]_1a(x)|^q  \, dx  \right)^{\frac{1}{q}} \\ 
& \lesssim \|b\|_{Lip_{\alpha}} \, r^{\alpha+n(1-\frac{1}{p})} \left(\int_{|x-x_0|> 2r} |x-x_0|^{-n q}  \, dx  \right)^{\frac{1}{q}} \\ 
& \quad + \|b\|_{Lip_{\alpha}} \, r^{n(1-\frac{1}{p})} \left(\sum_{k=1}^{\infty} \omega_{\infty} \left(\frac{1}{2^{k-1}} \right) \int_{2^k r < |x-x_0|\leq 2^{k+1} r} |x-x_0|^{(\alpha-n)q}  \, dx \right)^{\frac{1}{q}} \\ 
& \lesssim \|b\|_{Lip_{\alpha}} + \|b\|_{Lip_{\alpha}} r^{n(1-\frac{1}{p})} \sum_{k=1}^{\infty}\omega_{\infty} \left(\frac{1}{2^{k-1}} \right)(2^k r)^{\alpha-n+\frac{n}{q}} \\ 
& \lesssim \|b\|_{Lip_{\alpha}} \left( 1 +  \sum_{k=1}^{\infty} (2^k)^{\alpha-n+\frac{n}{q}} \omega_{\infty} \left(\frac{1}{2^{k-1}} \right) \right)\\
& \lesssim \|b\|_{Lip_{\alpha}} \left( 1 + \sum_{k=1}^{\infty}(2^k)^{\alpha}\omega_{\infty} \left(\frac{1}{2^{k-1}} \right) \right) \\ 
& \lesssim \|b\|_{Lip_{\alpha}}.
\end{align*}
This completes the proof of Theorem \ref{thm:commutator-b-Lipschitz-Hp-Lq-estimate}. 

%%%%%%%%%%%%%%%%%%%%%%%%%%%%%%%
%%%%%%%%%%%%%%%%%%%%%%%%%%%%%%%

\section*{Acknowledgements}
The first author was supported in part by the INSPIRE Faculty Fellowship from the DST, Government of India. The third author was supported by the Prime Minister's Research Fellowship (PMRF) from the Ministry of Education, Govetrnment of India.

%%%%%%%%%%%%%%%%%%%%%%%%%%%%%%%
%%%%%%%%%%%%%%%%%%%%%%%%%%%%%%%
%%%%%%%%%%%%%%%%%%%%%%%%%%%%%%%
%%%%%%%%%%%%%%%%%%%%%%%%%%%%%%%
%%%%%%%%%%%%%%%%%%%%%%%%%%%%%%%
%%%%%%%%%%%%%%%%%%%%%%%%%%%%%%%
%%%%%%%%%%%%%%%%%%%%%%%%%%%%%%%
%%%%%%%%%%%%%%%%%%%%%%%%%%%%%%%
%%%%%%%%%%%%%%%%%%%%%%%%%%%%%%%

\bibliographystyle{amsalpha}

\providecommand{\bysame}{\leavevmode\hbox to3em{\hrulefill}\thinspace}
\providecommand{\MR}{\relax\ifhmode\unskip\space\fi MR }
\providecommand{\MRhref}[2]{%
  \href{http://www.ams.org/mathscinet-getitem?mr=#1}{#2}
}
\providecommand{\href}[2]{#2}

\end{document}